\documentclass{article}

\usepackage[utf8]{inputenc}

\usepackage[utf8]{inputenc}
\usepackage{graphicx}
\usepackage[T1]{fontenc}
\usepackage{microtype}
\usepackage[full]{textcomp}
\usepackage{algorithm2e}   
\usepackage{amsmath,amssymb,amsthm}         
\usepackage{stmaryrd}
\usepackage{bm}
\usepackage{geometry}

\usepackage{xcolor}
\usepackage{dsfont}
\usepackage{stmaryrd}
\usepackage{nicefrac}
\usepackage{hyperref}
\hypersetup{
    colorlinks=true,
    linkcolor=blue,
    filecolor=magenta,    
    citecolor=purple,
    urlcolor=cyan,
    pdftitle={Randomized switching},
    pdfpagemode=FullScreen,
    }
\usepackage{amsmath,amssymb,dsfont,amsthm}
\usepackage{stmaryrd}

\newtheorem{theorem}{Theorem}
\newtheorem{proposition}{Proposition}
\newtheorem{lemma}{Lemma}
\newtheorem{corollary}{Corollary}[theorem]

\theoremstyle{definition}
\newtheorem{definition}{Definition}[section]
\newtheorem{assumption}{Assumption}

\theoremstyle{remark}
\newtheorem*{remark}{Remark}

\DeclareSymbolFont{sansops}{OT1}{\sfdefault}{m}{n}
\SetSymbolFont{sansops}{bold}{OT1}{\sfdefault}{b}{n}

\makeatletter
\renewcommand\operator@font{\mathgroup\symsansops}
\makeatother

\DeclareSymbolFont{sfoperators}{OT1}{cmss}{m}{n}
\DeclareSymbolFontAlphabet{\mathsf}{sfoperators}

\makeatletter
\def\operator@font{\mathgroup\symsfoperators}
\makeatother

\newcommand{\R}{\mathbb{R}}

\newcommand{\N}{\mathbb{N}}

\newcommand{\E}{\operatorname{ E}}

\newcommand{\vct}[1]{\boldsymbol{#1}}
\newcommand{\mtx}[1]{\boldsymbol{#1}}
\newcommand{\diag}{\operatorname{diag}}

\newcommand{\<}{\langle}
\renewcommand{\>}{\rangle}
\newcommand{\T}{\top}

\newcommand{\Null}{\operatorname{ ker}}
\newcommand{\Range}{\operatorname{ range}}
\newcommand{\Span}{\operatorname{ span}}

\newcommand{\trace}{\operatorname{ tr}}

\newcommand{\ip}[2]{\left\<#1, #2\right\>}
\newcommand{\norm}[1]{\left|\left|#1\right|\right|}
\newcommand{\opnorm}[1]{\norm{#1}_{ \operatorname{op}}}
\newcommand{\Expec}[2][]{\E_{#1}\left[#2\right]}

\newcommand{\p}[1]{\left(#1\right)}
\newcommand{\s}[1]{\left[#1\right]}

\renewcommand{\c}[1]{\left\{#1\right\}}
\newcommand{\abs}[1]{\left|#1\right|}
\newcommand{\Ind}[1]{\mathds{1}\left\{#1\right\}}

\newcommand{\set}[1]{\mathcal{#1}}

\DeclareMathOperator*{\argmin}{\operatorname{arg~min}}

\newcommand{\va}{\vct{a}}

\newcommand{\vd}{\vct{d}}
\newcommand{\ve}{\vct{e}}
\newcommand{\vf}{\vct{f}}

\newcommand{\vh}{\vct{h}}

\newcommand{\vp}{\vct{p}}

\newcommand{\vs}{\vct{s}}

\newcommand{\vu}{\vct{u}}
\newcommand{\vv}{\vct{v}}
\newcommand{\vw}{\vct{w}}
\newcommand{\vx}{\vct{x}}
\newcommand{\vy}{\vct{y}}
\newcommand{\vz}{\vct{z}}

\newcommand{\vbeta}{\vct{\beta}}

\newcommand{\vzeta}{\vct{\zeta}}

\newcommand{\vtheta}{\vct{\theta}}

\newcommand{\vrho}{\vct{\rho}}

\newcommand{\vzero}{\vct{0}}
\newcommand{\vone}{\vct{1}}

\newcommand{\mA}{\mtx{A}}
\newcommand{\mB}{\mtx{B}}
\newcommand{\mC}{\mtx{C}}

\newcommand{\mE}{\mtx{E}}

\newcommand{\mH}{\mtx{H}}

\newcommand{\mL}{\mtx{L}}
\newcommand{\mM}{\mtx{M}}

\newcommand{\mP}{\mtx{P}}

\newcommand{\mR}{\mtx{R}}
\newcommand{\mS}{\mtx{S}}

\newcommand{\mW}{\mtx{W}}
\newcommand{\mX}{\mtx{X}}
\newcommand{\mY}{\mtx{Y}}

\newcommand{\mGamma}{\mtx{\Gamma}}

\newcommand{\mPi}{\mtx{\Pi}}

\newcommand{\mId}{\mathbf{I}}

\newcommand{\vsbar}{\overline{\vs}}

\newcommand{\setB}{\set{B}}

\newcommand{\setD}{\set{D}}

\newcommand{\setP}{\set{P}}

\newcommand{\setS}{\set{S}}

\newcommand{\st}{\operatorname{\sf s.t. }}

\newcommand{\ow}{\operatorname{\sf otherwise }}
\newcommand{\sgn}{\operatorname{sign}}

\renewcommand{\S}{\mathbb{S}}

\renewcommand{\phi}{\varphi}
\renewcommand{\rho}{\varrho}
\renewcommand{\epsilon}{\varepsilon}
\renewcommand{\vv}{\vct{v}}

\renewcommand{\c}[1]{\left\{#1\right\}}

\newcommand{\vstilde}{\vct{\tilde s}}

\newcommand{\vell}{\vct{\ell}}
\newcommand{\vxtilde}{\vct{\tilde x}}

\newif\ifapdxthms
\apdxthmsfalse

\usepackage{wasysym}

\usepackage{mathtools}
\DeclarePairedDelimiter{\ceil}{\lceil}{\rceil}

\begin{document}

\title{Efficient Network Reconfiguration by Randomized Switching}

\author{Samuel Talkington\thanks{School of Electrical and Computer Engineering, Georgia Tech, Atlanta, GA, USA.\\ Web: \href{https://samueltalkington.com}{https://samueltalkington.com}, 
\href{https://molzahn.github.io}{https://molzahn.github.io}, 
Email: \{\href{mailto:talkington@gatech.edu}{talkington}, 
\href{mailto:molzahn@gatech.edu}{molzahn}\}@gatech.edu
}, \qquad Dmitrii M. Ostrovskii\thanks{Schools of Mathematics and Industrial and Systems Engineering, Georgia Tech, Atlanta, GA, USA.\\
Web: \href{https://ostrodmit.github.io/}{https://ostrodmit.github.io/}, Email: \href{mailto:ostrov@gatech.edu}{ostrov}@gatech.edu}, \qquad Daniel K. Molzahn$^*$
}
\date{\today}
\maketitle

\begin{abstract}
  We present an algorithm that efficiently computes nearly-optimal solutions to a class of combinatorial reconfiguration problems on weighted, undirected graphs. Inspired by societally relevant applications in networked infrastructure systems, these problems consist of simultaneously finding an unreweighted sparsified graph and nodal potentials that satisfy fixed demands, where the objective is to minimize some congestion criterion, e.g., a Laplacian quadratic form. These are mixed-integer nonlinear programming problems that are NP-hard, in general. To circumvent these challenges, instead of solving for a single best configuration, the proposed randomized switching algorithm seeks to design a distribution of configurations that, when sampled, ensures that congestion concentrates around its optimum. We show that the proposed congestion metric is a generalized self-concordant function in the space of switching probabilities, which enables the use of efficient and simple conditional gradient methods.  We implement our algorithm and show that it outperforms a state-of-the-art commercial mixed-integer second-order cone programming (MISOCP) solver by orders of magnitude over a large range of problem sizes.
\end{abstract}

\newpage

\section{Introduction}
\label{sec:intro}
Consider an undirected weighted graph~$G = (N,E,\vw)$ with $n = \abs{N}$ nodes and~$m=\abs{E}$ \emph{switchable edges} which can be opened or closed. The edge weights $\vw \in \R^m_+$ are fixed. We are given a fixed vector of nodal demands $\vd \in \R^n$ that satisfies $\vd^\T \vone = 0$. A \emph{network reconfiguration problem} is to find edge \emph{switching statuses} $\vs \in S \subseteq \c{0,1}^m$ \emph{and} nodal potentials $\vx \in \R^n$ that minimize a \emph{congestion criterion}, subject to various engineering constraints. For example, such constraints may include ensuring that the graph is connected and that the nodal potentials satisfy network physics for the given demands over the switched graph. A network reconfiguration problem can be formulated as the following parametric mixed-integer nonlinear program (MINLP):
    \begin{align*}
        \setP(\vd) :\qquad \min_{\vs,\vx}& \qquad  \vx^\T \mGamma_{\vs} \vx \tag{Congestion}\label{eq:prob:energy}\\
        \st& \qquad \mL_{\vs}\vx = \vd \tag{Physics}\label{eq:prob:phys}\\
        &\qquad \lambda_2(\mL_{\vs}) \geq \lambda_\star \tag{Connectivity}\label{eq:prob:conn}\\ 
        & \qquad \vs \in S \tag{Switching}\label{eq:prob:switching}\\
        & \qquad \vx\perp\vone \tag{Kirchhoff},
    \end{align*}   
    where $\mL_{\vs} := \sum_{i<j} \mE_{ij}w_{ij}s_{ij}$ is the \emph{model Laplacian} of the switched graph $G_{\vs} := (N,E,\vw \circ \vs)$, and $\mGamma_{\vs} := \sum_{i<j} \mE_{ij} \gamma_{ij} s_{ij}$ is a \emph{target Laplacian} for a graph with the same connectivity, but potentially different edge weights. The matrix~$\mE_{ij} := \p{\ve_i - \ve_j}\p{\ve_i - \ve_j}^\T$ is an elementary Laplacian describing a two-node graph. The constraint on the second-smallest eigenvalue~$\lambda_2(\mL_{\vs}) \geq \lambda_\star$  ensures the graph remains connected.
    
    There are several particularly relevant cases for the objective \eqref{eq:prob:energy} and the target Laplacian $\mGamma_{\vs}$, with natural connections to engineering applications, such as in electric power systems:
    \begin{enumerate}
        
        \item the original \emph{model Laplacian}, i.e., $\gamma_{ij} = w_{ij}$ for all~$(i,j) \in E$. In this case, if $w_{ij}$ represents the \textit{conductance} of edge~$(i,j) \in E$, the objective becomes the \emph{total electric power dissipated} across all nodes in the network~$\sum_{i=1}^n d_{i} x_{i} := \sum_{i=1}^n p_i$.
        \item  a \emph{capacitated Laplacian}, e.g.~$\gamma_{ij} = w_{ij}^2/c_{ij}^2$, for all~$(i,j) \in E$. In this case, the objective function becomes the sum of the normalized line flows $\sum_{i<j} f_{ij}^2/c_{ij}^2$.
        \item a  \emph{risk}, \emph{fairness}, or \emph{robustness} functional, which measures the importance of an edge with respect to some metric. For instance,~$\gamma_{ij}$ could represent the probability that a wildfire ignites at a line in an electrical network, or the criticality of the edge in socioeconomic factors. 
    \end{enumerate}
    These cases can be derived in a straightforward manner from the general framework of electrical networks and their corresponding graph representations.

\subsection{Motivation}
\label{sec:intro:motivation}
The program~$\setP$ is inspired by more complex network reconfiguration problems that are of great current interest to electric power system researchers; see \cite{Babaeinejadsarookolaee_congestion_2023,haider_siphyR_2025,vandersar2025optimizingpowergridtopologies} for recent examples. Such problems are known as \emph{transmission switching}, \emph{topology control}, or \emph{network reconfiguration} in the relevant literature. The constraint \eqref{eq:prob:phys} is identical to \emph{DC power flow}, which is a linear approximation of the nonlinear AC power flow equations. Even our primitive problem~$\setP$ is NP-hard in general (see \cite{lehmann2014complexitydcswitchingproblems,Kooij_2023} for a related discussion), and the objective function is nonconvex.

A natural first attempt to solve the problem is to relax the switching variables $\vs$ to be continuous, i.e., $\vs \in [0,1]^m$, and then solve the resulting optimization problem. However, this approach does not yield a feasible solution to the original problem, since the relaxed solution need not be integral. Moreover, the relaxed problem is still a challenging problem, as~\eqref{eq:prob:phys} involves bilinear equality constraints.

\paragraph{Applications of interest.}
The present paper proposes a new primitive blending randomized network design and mixed-integer quadratic optimization. The proposed model has direct applications to widely used approximate models for electric power flow and water flow in real-world infrastructure networks, both of which obey the conditions of the proposed theory. Moreover, this primitive is directly inspired by the electric power transmission network reconfiguration problem, where transmission lines are switched on and off to improve the stability of the network. The algorithm is designed to be usable for the \textit{DC power flow approximation} of the nonlinear AC power flow equations, a canonical linearization that appears frequently in practice; see~\cite{stott_dc_revisited_2009} for a salient review. This approximation has the attractive problem of being \textit{identical} to the Poisson problem. In essence, we are given a \textit{net power injection vector}~$\vp \in \R^n$ such that~$\vp^\T \vone = 0$ and we want to find a \textit{voltage phase angle} vector~$\vtheta \in (-\pi,\pi]^n$ such that
\[
\vp = \mL \vtheta, \qquad \mL := - \mB, \qquad \vtheta^\T \vone =  0,
\]
where~$\mB \preceq 0$ is a weighted graph Laplacian matrix corresponding to an undirected graph with negative edge weights~$b : N \times N \to \R_-$, known as \textit{susceptances}.

While the proposed theory is a general graph algorithm, it is also applicable to electric power engineering, which has thus far taken largely empirical and deterministic mixed-integer-based approaches. By targeting approximate preservation of Poisson solutions with controllable energy reduction through randomized edge selection, we provide a new, truly interdisciplinary contribution to this problem.

\subsection{Related work}
Our work contributes to a rich literature in spectral graph theory; in particular, the \emph{effective resistance} of graphs, which has been studied extensively in recent decades. It has shown promise for improving the resilience of complex networks \cite{wang_improving_2014} and improving connectivity \cite{ghosh_growing_2006}. A particularly meaningful metric is the \emph{total effective resistance} of a graph, which is defined as the sum of the effective resistances of all edges in the graph. The total effective resistance is a convex function of the edge weights \cite{ghosh_minimizing_2008}, and it has found numerous applications.

    \paragraph{Graph algorithms.} Researchers have made numerous recent breakthroughs in fast algorithms for problems involving massive graph structures; see \cite{teng_laplacian_2010} for a central historical review. Some particularly influential recent results are in solving Laplacian systems \cite{kyng_approx_gaussian_2016}, maximum flow, and minimum cost flow \cite{chen_almostlinear_maxflow_2022}. These algorithms are unified in spirit by incorporating a degree of randomness. Similar to parallel independent work in \cite{zhou_efficient_2025,brown_fast_2024} and previous related work in \cite{ghosh_growing_2006,li_maximizing_trees_2020}, we consider the setting of a graph with $m$ potential edges, where we wish to add $q \ll m$ such edges to design a favorable connected graph. This is related to the problem of maximization of algebraic connectivity, that is, maximization of the second-smallest eigenvalue of the Laplacian matrix \cite{brown_fast_2024,ghosh_growing_2006}, and also equivalent to maximization of the smallest eigenvalue of a grounded Laplacian matrix \cite{zhou_maximizing_2025}. This a classic question in spectral graph theory has been explored from various angles. The work of \cite{li_maximizing_trees_2020} viewed this through the lens of maximizing the number of spanning trees, and \cite{ghosh_growing_2006} through the lens of optimizing over edge weights lying on the simplex, and in \cite{zhou_maximizing_2025} through the lens of adding and deleting rows of the Laplacian. This was addressed through the lens of edge addition in \cite{ru_maximizing_2025}. The problem of maximizing the algebraic connectivity $\lambda_2$ can be viewed as seeking to improve global connectivity and robustness properties, which has found numerous applications in engineering \cite{LOpt_ECC2015,LOpt_TCNS2024,somisetty_spectral_robots_2024}. Our work complements this rich literature by focusing on the context of \emph{fixed} demands and edge weights; this problem setting introduces distinct challenges and opportunities.

    \paragraph{Spectral graph theory.} Effective resistances are one of the many useful results in the field of spectral graph theory. Multiple results related to this concept have inspired the present paper. In particular, the seminal result of~\cite{spielman_graph_2011}, which produced a randomized procedure for spectral sparsification\textemdash first defined in~\cite{spielman_teng_2004}\textemdash has heavily influenced our work. The extension of the spectral sparsification framework to the case where the graph edges cannot be augmented by the algorithm, i.e., unweighted sparsifiers like~\cite{avrom_subset_2013,anderson2014efficientalgorithmunweightedspectral}, are highly related to our work. Moreover, our notion of approximate optimality to the optimization problem under study is equivalent to that of spectral approximation, as in~\cite{spielman_spectral_2011,spielman_graph_2011}, and in other work in covariance estimation~\cite{pmlr-v23-hsu12,pmlr-v99-ostrovskii19a}.

    Spielman and Teng explored algorithms for fast solutions to Laplacian systems of equations~\cite{spielman_nearly-linear_2012}. More recent advances, including works by Kyng and collaborators, emphasize randomized algorithms for graph-based linear solvers and approximations \cite{kyng2016approximategaussianeliminationlaplacians}, which seems to yield advantages in theoretical simplicity and practical speed \cite{gao_robust_2023}.

    Deriving concentration inequalities for quantities in algebraic and spectral graph theory is one of the most promising applications of random matrix theory \cite{chen_spectral_2021}. Oliveira's investigations of concentration inequalities for Laplacians \cite{oliveira_concentration_2010} was an early result in this direction, which gave rise to the matrix Freedman inequality \cite{tropp2011freedmansinequalitymatrixmartingales}. Similarly, interesting results have been observed for the concentration of the total effective resistance \cite{boumal_concentration_2014}, also known as the Kirchhoff index.    A key component of our analysis relates to the literature on concentration inequalities for the minimum eigenvalues of positive-semidefinite matrices, e.g., see \cite{Tropp_2011,oliveira2013lowertailrandomquadratic}. Previously, Brown, Laddha, and Singh developed a framework to maximize the minimum eigenvalues of the sums of rank-one matrices \cite{brown_fast_2024}, which also relied on matrix concentration.

    \paragraph{Randomized algorithms.}
    Our work relates broadly to randomized algorithms, an in particular to their intersection with spectral graph theory. The two fields have deep connections. Foundational graph sparsification techniques, such as those developed by Spielman and Srivastava~\cite{spielman_graph_2011}, seek to approximate a Laplacian operator by a sparse subgraph while preserving spectral properties globally. Graph sparsification has found many applications such as regularization in machine learning~\cite{pmlr-v51-sadhanala16}, quantum computing~\cite{moondra2024promisegraphsparsificationdecomposition}, and others. Numerous other approaches to this problem, such as approximate matrix multiplication~\cite{charalambides_matmul_sparselap_2023} have been proposed. In addition, the authors of~\cite{lau_spectral_2020} gave an iterative randomized rounding algorithm approach to spectral network design. This is related to the problem of randomized experimental design~\cite{allen-zhu_near-optimal_2021}. We add to this literature by focusing on randomized network reconfiguration task aimed at controlling a Laplacian quadratic form under a specific solution profile induced by a fixed demand vector. This reveals previously unknown structural properties about the objective, and efficient algorithms for optimizing it.

    \paragraph{Algorithmic tools.} 
    One of our analyses is based on a Frank-Wolfe type procedure; see \cite{pokutta_frank-wolfe_2024} and \cite{braun_conditional_2025} for recent surveys. This class of first-order iterative convex optimization algorithms are particularly useful in the setting where projection steps are inefficient. Frank-Wolfe methods have found numerous applications, ranging from traffic assignment problems in transportation networks \cite{FUKUSHIMA1984169}, to recent SDP solvers \cite{yurtsever_scalable_2021,pham_scalable_fw_2023}.  


\section{Preliminaries}
\label{sec:prelim}
The following parameterized positive-semidefinite (PSD) Laplacian matrix functional is the central object of study.
\begin{definition}[Switched Laplacian]
    \label{def:sw_lap}
    For a multigraph $G = (N,E,\vw)$, let $\mA \in \c{-1,0,1}^{m \times n}$ be the node-to-edge incidence matrix for $G$. Let $\vs \in \c{0,1}^m$ be \emph{switching variables} corresponding to a reconfiguration, and let $\mL_{(\cdot)} : \c{0,1}^m \to \S^{n \times n}$ be a \textit{switched Laplacian matrix} functional, where
    \begin{equation}
    \label{eq:switched_laplacian}
    \mL_{\vs} := \mA^\T \mW^{1/2}\mS\mW^{1/2}\mA,
    \end{equation}
    corresponds to the switched graph~$G_{\vs} := (N,E,\vw \circ \vs)$,
    with~$\mS := \diag(\vs)$.
\end{definition}
Def. \ref{def:sw_lap} is also known as an \textit{unweighted sparsifier} \cite{anderson2014efficientalgorithmunweightedspectral}. From here, we can naturally define corresponding effective resistance and leverage score functionals.
 \begin{definition}[Effective resistance and leverage]
    \label{def:eff_res_and_leverage}
        The effective resistance $\rho_{ij} : \c{0,1}^m \to \R_+$ of a multi-edge $ij \in E$ under switching strategy $\vs$ is defined as
        \[
        \rho_{ij}(\vs) := \ve_{ij}^\T \mL_{\vs}^\dagger \ve_{ij}.
        \]
        Moreover, let $w_{ij} \in \R_+$ be the weight of the edge $ij$. The \emph{leverage} of the edge $ij$ is defined as the effective resistance scaled by the weight:
        \[
        \ell_{ij}(\vs) := w_{ij} \cdot \rho_{ij}(\vs) \leq 1.
        \]
        Note that one has $\sum_{ij \in E} \ell_{ij}(\vs) = n-1$. This is also known as Foster's theorem \cite{foster1949average}.
    \end{definition}

    \subsection{Relaxed problem formulation}
    \label{sec:prelim:simplified}
        Throughout, we work under the following simplifying assumptions. Assumption \ref{assum:simple-targets} is without loss of generality, while Assumption \ref{assum:backbone} is a somewhat restrictive assumption that provides a clean way to handle connectivity constraints.
        \begin{assumption}
        \label{assum:simple-targets}
        The target edge weights are the same as the model weights, $y_e = w_e$ for all $e \in E$; equivalently, $\mGamma_{\vs} = \mL_{\vs}$ for any $\vs \in \c{0,1}^m$.
        \end{assumption}
        \begin{assumption}
            \label{assum:backbone}
            There exists a connected template graph $S_0 \subseteq E$ such that $s_e = 1$ always for all $e \in S_0$.
        \end{assumption}

        \paragraph{Relaxation to switching probabilities.} The key step is to relax the binary variables to lie in the unit cube $\vs \in \s{0,1}^m$. This choice has two benefits. From an optimization perspective, the relaxed reconfiguration program becomes much more tractable. Moreover, fractional solutions of the relaxed reconfiguration problem can be viewed as \textit{switching probabilities}, through the framework of randomized rounding~\cite{raghavan_randomized_1987}.
        Under Assumptions \ref{assum:simple-targets} and \ref{assum:backbone}, by primal feasibility, the reconfiguration problem can be written as
        \begin{align*}
            \min_{\vs} \qquad &\vd^\T \mL_{\vs}^\dagger \vd \;=: \phi(\vs)\\
            \st \qquad &s_e = 1 \quad \forall e \in S_0,\\
            \qquad &\norm{\vs}_1 \leq q.
        \end{align*}

        The relaxed congestion objective $\phi : \s{0,1}^m \to \R_+$ can be written as 
        \begin{equation}
        \label{eq:relaxed_cong}
        \phi(\vs) = \vd^\T \mL_{\vs} \vd, \qquad \mL_{\vs} = \mA^\T \diag(\vw \circ \vs) \mA,
        \end{equation}
        and it has notable structural properties, which we now investigate.

\subsection{First-order information}
\label{sec:first-second}

The gradient of the relaxed congestion~\eqref{eq:relaxed_cong} can be computed in closed form. To obtain this, we first show how to differentiate a related function, the \textit{total effective resistance} or \textit{Kirchhoff index}, which was studied by Ghosh and Boyd in~\cite{ghosh_minimizing_2008}. 
\begin{lemma}
    \label{lem:ghosh_gradient}
    Let~$R(\vs) = \sum_{i<j} \varrho_{ij}(\vs) = \frac{1}{2} \vone^\T \mR_{\vs} \vone$ be the total effective resistance of the graph~$G_{\vs} = (N,E,\vw \circ \vs)$. Then, 
    \[
    \frac{\partial}{\partial s_{ij}} R(\vs) = -n w_{ij} \ve_{ij}^\T\mL^{2\dagger}_{\vs} \ve_{ij}, \qquad ij \in E.
    \]
\end{lemma}

In particular, in our setting of fixed demands, we generalize the above result to the following.
\begin{lemma}[Congestion gradient]
\label{lem:switching_gradient}
    For a fixed demand $\vd \perp \vone$, the gradient of the  relaxed congestion \eqref{eq:relaxed_cong} is given elementwise for each edge $e \in E$ as
    \begin{equation}
        \label{eq:congestion_gradient_entries}
        \frac{\partial}{\partial s_e}\,\phi(\vs) = - w_e \abs{\ip{\va_e}{\mL_{\vs}^\dagger \vd}}^2 := -w_e\Delta_e^2 \leq 0, 
    \end{equation}
    where $\Delta_e = \hat{x}_i - \hat{x}_j$ is the voltage difference \textit{solution map} across edge $e=(i,j) \in E$. 
\end{lemma}
Using the gradient entries of the relaxed congestion~\eqref{eq:congestion_gradient_entries}, we will show that the relaxed congestion~\eqref{eq:relaxed_cong} inherits many of the favorable properties of the total effective resistance.

\subsection{Second-order information and Hessian}
\label{sec:hessian}
    To analyze the performance of our algorithms, we need to understand the spectral content of the Hessian of the congestion function $\phi(\vs)$ with respect to the switching variables $\vs$. We will first state and prove the following technical lemma.
    \begin{lemma}[Cauchy–Schwarz for voltage differences]
        \label{lemma:voltage_cauchy_schwarz}
        Let $\vd \in \R^n$ be a vector of demands, and let $ \vx = \mL_{\vs}^\dagger \vd$ be the voltages induced by the demands under switching strategy $\vs \in \c{0,1}^m$. Then, for any edge $ij \in E$ and switching strategy $\vs$, we have
        \begin{equation}
            \label{eq:voltage_cauchy_schwarz}
            \Delta_{ij}^2 = \abs{\va_{ij}^\T \vx}^2  \leq \rho_{ij} \cdot \phi(\vs) \qquad \text{for all } (i,j) \in E.
        \end{equation}
    \end{lemma}
    Lemma \ref{lemma:voltage_cauchy_schwarz} says that the squared voltage differences in electrical networks can never exceed the product of the effective resistance of the edge and the total congestion in the network. This is a consequence of the fact that $\mL_{\vs},\mL_{\vs}^\dagger \succeq 0$, and the Cauchy-Schwarz inequality applied to the voltages induced by the demands $\vd$ on the switched graph $G_{\vs} = (N,E,\vw \circ \vs)$. 

    \begin{lemma}[Operator norm bound on the Hessian]
        \label{lemma:hessian_operator_norm_bound}
        Given a switching strategy set $\setS \subseteq \s{0,1}^m$ and demands~$\vd$, let $\mH(\cdot) : \setS \to \S^m_+$ be the Hessian of the congestion function $\phi(\cdot)$. The operator norm of the Hessian $\mH(\cdot)$ is bounded as follows:
        \begin{align*}
        \opnorm{\mH} &\leq L := 2\max_{\vs \in \setS} \phi(\vs) 
        \leq 2\norm{\vd}_{\mL_{\vs_0}}^2
        \leq \frac{2 \norm{\vd}_2^2}{\lambda_2\p{\mL_{\vs_0}}}.
        \end{align*}
        Moreover, if $\norm{\vd}_2 \leq 1$ always, and the template is such that $\lambda_2(\mL_{\vs_0}) \geq 1$, then it holds that
        $ \opnorm{\mH} \leq 2$. 
        
    \end{lemma}
    The proof appears in the supplementary material.

\subsection{Structural properties of congestion}
\label{sec:gsc-lemma}

The goal of this section is to outline useful properties about the relaxed congestion \eqref{eq:relaxed_cong}.
\begin{remark}
    Set $\vone_T\in\{0,1\}^m$ as the indicator of a spanning tree $T\subseteq E$ of $G$.  In the feasible set
    \begin{equation*}
        \mathcal{D}_T \;:=\; \{\,\vs\in[0,1]^m : s_e=1 \text{ for all } e\in T\},
    \end{equation*}
    the graph remains connected for all $\vs\in \mathcal{D}_T$ because the tree edges have switching value one.
\end{remark}
Furthermore, we have the following useful property, which is reminiscent of \cite{ghosh_minimizing_2008}.
\begin{lemma}
\label{lemma:homogeneity}
The relaxed congestion \eqref{eq:relaxed_cong} is a convex, homogeneous function of degree $-1$, namely,
\[
\phi(\vs) = -\ip{\nabla \phi(\vs)}{\vs} \qquad \forall \vs \in \setD_T.
\]
\end{lemma}
The proof appears in Appendix~\ref{apdx:lemma:homogeneity}.

 \begin{theorem}[Generalized self\nobreakdash--concordance]
    \label{thm:backbone-gsc}
    Let $T \subseteq E$. For $\vs \in\mathcal{D}_T$ define the relaxed congestion~$\phi(\vs)=\vd^\T \mL_{\vs}^\dagger \vd$.  Let~$\vrho_T := \diag(\mA \mL_T^\dagger \mA^\T)\in\R^m$ denote the vector of effective resistances in~$T$, i.e.,~$\vrho_T(e)=\va_e^{\top}\mL_T^{\dagger}\va_e$ where~$\mL_T=\mA^{\top}\diag(\vw\odot \vone_T)\mA$ is the Laplacian corresponding to~$T$ and~$\mL_T^\dagger$ its pseudoinverse.  Then~$\phi$ is~$(M,\nu)$--generalized self\nobreakdash--concordant on~$\mathcal{D}_T$ with
    \begin{equation}
        \nu=2\quad\text{and}\quad M\;=\;3\,\norm{\vw\odot \vrho_T}_2.
    \end{equation}
    That is, for all $s\in\mathcal{D}_T$ and all directions $\vu,\vv\in\R^m$ one has
    \begin{equation}
        \bigl|D^3\phi(s)[\vv, \vu, \vu]\bigr|\;\le\;3\,\bigl\|\vw\odot \vrho_T\bigr\|_2\,\|\vv\|_2\,\|\vu\|_{\nabla^2 \phi(\vs)}^2.
    \end{equation}
    In particular, the generalized self\nobreakdash--concordance constant $M$ is finite and independent of $\vs$ on the entire feasible domain $\mathcal{D}_T$.
\end{theorem}
The proof appears in Appendix~\ref{apdx:proof:backbone-gsc}.


\section{Main Results}
\label{sec:main_results}

    \subsection{Approximation Primitives}
    \label{sec:algorithms}

    As shown in Thm.~\ref{thm:backbone-gsc}, $\phi$ is a generalized self-concordant function in the space of \emph{switching probabilities} $\vs \in \setD_T$. This makes a particularly simple variant of the Frank-Wolfe algorithm a good choice to design a distribution of configurations; see the recent works of~\cite{dvurechensky_generalized_2023,pham_scalable_fw_2023,carderera2024scalable}. We aim for our random output switching vector to be such that, when we sample edges without replacement (that is, \emph{draw a configuration}), it is likely to produce a connected subgraph with at most~$O(q)$ edges, and to satisfy~$d^\top L^\dagger_s d \leq \alpha \cdot d^\top L^\dagger_{\mathbf{1}_m} d$, where we desire~$\alpha$ to be as small as possible. To achieve an efficient implementation of this algorithm, in this section, we will present the key ingredients that leverage fast Laplacian solvers.

    \subsubsection{Fast Laplacian solver with switching}
    \label{sec:fast_solver}
        Below is a simple corollary of the result in \cite[Thm.~1.2.1.]{kyng_dissertation}, which proves useful in our subsequent analysis.
        \begin{corollary}[\cite{kyng_dissertation}]
            \label{cor:kyng:laplacian_solver}
            Fix a scalar $\delta <1/n^{100}$ and a switching strategy $\vs$ with at most $q$ non-zero entries, and let $\mL_{\vs} \succeq \vzero$ be as in Def.~\ref{def:sw_lap}. There exists an algorithm that returns a random Cholesky factor $\mC_{\vs}$ such that, with probability at least $1-O(\delta)$,
        \[
        \tfrac{1}{2}\mL_{\vs} \preceq \mC_{\vs}\mC_{\vs}^\T \preceq \tfrac{3}{2} \mL_{\vs}.
        \]
        This algorithm runs in time~$O\p{q\log^2\p{1/\delta}\log(n)}$. 
    
        Consequently, given demands~$\vd \perp \vone$ and an error tolerance~$\epsilon >0$ there is an algorithm $\mathsf{Solve}(\mL_{\vs},\vd,\delta,\epsilon)$ that returns an approximate voltage solution $\hat\vx(\vs)$ such that with probability at least $1-O(\delta)$, 
        \[ 
        \norm{\hat{\vx}(\vs) - \mL^\dagger_{\vs} \vd}_{\mL_{\vs}} \leq \epsilon \norm{\mL_{\vs}^\dagger \vd}_{\mL_{\vs}},
        \]
        where~$\norm{\vx}_{\mL} := \sqrt{\vx^\T \mL \vx}$ for $\vx \in \R^n$. 
        The algorithm runs in~$O(q \log^2(1/\delta)\log(n)\log(1/\epsilon))$ time.
        \end{corollary}

    Access to a near-linear time Laplacian solver immediately provides a simple, near-linear time algorithm for \textit{simultaneously} determining the gradient of the congestion criterion~$\nabla \phi \in \R^m$, and pairwise voltage differences across all edges,~$\Delta \in \R^m$, which we present in Section \ref{sec:main:fast_grad}.

    \subsubsection{Switching strategies and edge budget}
    \label{sec:prob:switching_strategies}
    
    A useful reconfiguration problem is to select a subset of edges $S \subseteq E$ such that the graph remains connected, i.e., the algebraic connectivity $\lambda_2(\mL_S) > 0$, the number of edges in $S$ is at most $K$, and the congestion (graph energy) does not grow too much. There are two such ways we can model this constraint set. 

    Throughout, we treat all switching variables as the parameters of independent Bernoulli random variables, i.e., we assume that each edge $e \in E$ is switched with probability $s_e \in [s_{\sf min},1]$, with the lower bound to be defined later, and the switching decisions are independent. We want to show that $\vs$ is a good \emph{approximately optimal switching strategy}, in the following sense: if we sample a {\em random} configuration~$\vstilde$ with~$\vstilde_{e} \sim \mathsf{Ber}(\vs_{e})$, independently for each edge~$e \in E$, then the congestion of the resulting graph does not grow too much. We will define this notion precisely in Theorem \ref{thm:bern_rounding_main}.
    
    A natural constraint is to impose a budget on the number of edges that can be switched, that is, we require that~$\norm{\vs}_1 \leq q$, where $q$ is the edge budget. This is a natural constraint in many applications, such as electrical networks, where we may not be able to switch all edges due to physical limitations or operational costs.

    \subsubsection{Fast gradients}
    \label{sec:main:fast_grad}

    We now study the algorithm \textsc{ApproxDiff}, presented in Alg.~\ref{alg:gradient_compute}, which updates the congestion gradient and the voltage differences for a given switching strategy $\vs$ and demands $\vd$ in nearly-linear time.
    \begin{lemma}[Approximate gradient]
        \label{lemma:approximate_gradient}
        Let $\vs \in [0,1]^m$ be a switching strategy, and let $\vd \perp \vone$ be demands. Fix $\delta,\epsilon \in (0,1)$. Let
        \[
        \hat{\vx}(\vs)\ \gets\ \mathsf{Solve}(\mL_{\vs},\vd,\delta,\varepsilon)
        \]
        be a randomized Laplacian solve which, with probability at least $1-O(\delta)$, returns $\hat{\vx}$ satisfying
        \[
        \|\hat{\vx}-\vx\|_{\mL_{\vs}}\ \le\ \varepsilon\,\|\vx\|_{\mL_{\vs}},
        \qquad \text{where } \vx=\mL_{\vs}^{\dagger}\vd .
        \]
        Define the approximate gradient entrywise by
        $\widehat{\nabla}\phi(\vs)_e := -\,w_e\,(\va_e^\top\hat{\vx})^2$,
        and compute all entries with a single call to \textsc{Solve} via \textsc{ApproxDiff}.
        Then \textsc{ApproxDiff} runs in time~$O(m\log^c(n)\log(1/\epsilon))$, and there exists a constant $C$
        such that, with probability at least $1-O(\delta)$,
        \[
        \bigl\|\widehat{\nabla}\phi(\vs)-\nabla\phi(\vs)\bigr\|_{\nabla^2\phi(\vs)^\dagger}
        \ \le\ C\,\varepsilon\ \bigl\|\nabla\phi(\vs)\bigr\|_{\nabla^2\phi(\vs)^\dagger}.
        \]
    \end{lemma}
    The proof appears in Appendix~\ref{apdx:approx:grad}. 
    \begin{algorithm}[t]
    \label{alg:gradient_compute}
    \small
      \caption{Fast gradient and voltage difference computation for network reconfiguration.}
      \DontPrintSemicolon
      \KwIn{$\vs$, $\vd$, $\vw$, $E$, $\delta$, $\epsilon$.}
      \KwOut{Approximations $\hat{\nabla}\phi(\vs)$ and $\hat{\Delta}(\vs)$.}
      \BlankLine
      
      \SetKwBlock{Begin}{function}{end function}
        \Begin($\textsc{ApproxDiff} {(} \vs,\vd,\vw,E,\delta,\epsilon {)}$)
        {
            $\mL_{\vs} \gets \sum_{ij \in \mathsf{supp}(\vs)} \mE_{ij} w_{ij}$ \tcp*{$O(m)$}
            $\nabla \gets \vzero_m,\; \Delta \gets \vzero_m$\;
            \tcp{--- Solve the Laplacian system ---}
            $\hat\vx(\vs) \gets \mathsf{Solve}\p{\mL_{\vs},\vd,\delta}$ \tcp*{$\tilde{O}(m)$}
            \tcp{--- compute the gradient entries ---}
            \For{$e=(i,j) \in E$}{
                \tcp{--- compute voltage diff. ---}
                $\Delta_{e} \gets \hat{x}_i(\vs) - \hat{x}_j(\vs)$ \tcp*{$O(1)$}
                \tcp{--- compute gradient entry ---}
                $\nabla_{e} \gets -w_{ij} \Delta_{e}^2$ \tcp*{$O(1)$}
            }
        \Return{$\nabla,\Delta$}
        }
    \end{algorithm}

    \begin{remark}
        In practice, Algorithm \ref{alg:gradient_compute} should be supplied a persistent copy of $\nabla,\Delta\in \R^m$ for in-place storage of the congestion gradient and the voltage differences.
    \end{remark}




\begin{algorithm}[t]
\label{alg:switching}
    \small
  \caption{Montonic Frank-Wolfe with randomized rounding for network reconfiguration.}
  \DontPrintSemicolon
  \KwIn{$\vs_0$, $\vd \perp \vone$, $E$, $\vw$, $q$, $T$, $\epsilon$, $\delta$.}
  \KwOut{Random integral switching strategy $\vstilde$.}
  \BlankLine
  
  \SetKwBlock{Begin}{function}{end function}
    \Begin($\textsc{RandReconfig} {(} T,E,q,\vw,\vd,\epsilon,\delta {)}$)
    {
        $\mL_{\vs_0} \gets \sum_{ij \in \mathsf{supp}(\vs_0)} \mE_{ij} w_{ij}$ \tcp*{$O(m)$}
      \For{$t= 0,1,\ldots,$ \KwTo $T-1$}{
        $\eta_t \gets \frac{2}{t+2}$ \tcp*{Set step size}
        \tcp{--- compute the congestion gradient ---}
        $\nabla \phi(\vs_t),\Delta(\vs_t) \gets \textsc{ApproxDiff}(\vs_t,\vd,\vw,E,\epsilon,\delta)$\; 
        \tcp{--- find a vertex ---}
        $\vv_t \gets \argmin_{\vv \in \s{0,1}^m,\norm{\vv}_1 \leq q} \ \ip{\nabla \phi(\vs_t)}{\vv}$\;
        \tcp{--- update convex combination ---}
        $\vs_{t+1} \gets (1-\eta_t)\vs_t + \eta_t \vv_t$\;
      }
        \tcp{--- round the fractional solution ---}
        $\vstilde \gets \textsc{Round}(\vs_T)$ \tcp*{$\tilde{O}(m)$}
        \tcp{--- solve for the voltages ---}
        $\vxtilde \gets \mathsf{Solve}\p{\mL_{\vstilde},\vd,\delta,\epsilon}$\;
    \Return{$\vs_\star,\vxtilde$}
    }
\end{algorithm}

\subsection{Frank-Wolfe with randomized rounding}\label{sec:bern_round}

We apply a simple Frank-Wolfe procedure in Alg. \ref{alg:switching} to optimize the relaxed congestion function. After the iterative phase of the algorithm, we obtain a fractional switching vector $\vs_t\in[0,1]^m$ satisfying $\|\vs_t\|_1\le q$ and $(\vs_t)_e=1$ for all backbone edges $e\in T$.  To convert $\vs_t$ into an integral solution while preserving the edge weights $\vw$, we adopt the following procedure.  Fix a failure probability $\delta\in(0,1)$ and choose a baseline probability, e.g.~$p_{\min} \;=\; C\,\log(n/\delta)/n$ with a universal constant $C>0$.  Define 
\begin{equation}
\label{eq:sw_prob_def}
\bar{s}_e = \begin{cases}
    \max\{(\vs_t)_e,\,p_{\min}\} & e\in E\setminus T\\
    1 & e \in T.
\end{cases}
\end{equation}
Sample a random vector $\tilde{\vs}\in\{0,1\}^m$ by including each edge $e$ independently with probability $\bar{s}_e$.  If $\|\tilde{\vs}\|_1 > q$, remove edges in increasing order of $(\vs_t)_e$ until at most $q$ edges remain; if $\|\tilde{\vs}\|_1 < q$, add edges in decreasing order of $(\vs_t)_e$ until $q$ edges are selected. Another approach is to sample multiple draws of~$\vstilde$ until~$\|\vstilde\|_1\leq q$, without ensuring the constraint is tight. Scalar Bernoulli concentration can show that this will take at most a logarithmic number of draws. Denote by $\mL_{\vstilde}=\sum_e \tilde{s}_e w_e \va_e\va_e^\T$ the Laplacian of the sampled configuration.

\begin{theorem}[Rounding preserves congestion]\label{thm:bern_rounding_main}
Let $\vs_t\in[0,1]^m$ be the output of the $t$-th iterate of the Frank-Wolfe procedure, which satisfies $\|\vs_t\|_1\le q$ and $(\vs_t)_e=1$ for all $e\in T$. Define~$\bar{s}_e := \max\c{\p{\vs_t}_e,p_{\sf min}}$ for all $e \notin T$, and~$\bar{s}_e =1$ for $e \in T$. Obtain~$\tilde{\vs} \in \c{0,1}^m$ by independent Bernoulli rounding with $\Pr\s{\tilde{s}_e =1}=\bar{s}_e$. Then, there exists an absolute constant $C>0$ such that for any $\delta \in (0,1)$, the following holds with probability at least $1-\delta$:
\begin{equation}\label{eq:bern_spectral}
  (1-\epsilon)\,\mL_{\vsbar} \;\preccurlyeq\; \mL_{\tilde{\vs}} \;\preccurlyeq\; (1+\epsilon)\,\mL_{\vsbar}, \qquad \mathsf{on} \; \Span(\vone)^\perp,
\end{equation}
where 
\[
\epsilon = C \sqrt{\frac{R\log((n-1)/\delta)}{\lambda_2(\mL_{\vsbar})}}, \qquad \mathsf{with} \qquad R = 2 \cdot \max_e w_e.
\]
Consequently, for any demand vector $\vd\perp\vone$, one has
\[
 \frac{1}{1+\epsilon} \phi(\vsbar) \leq \phi(\vstilde) \leq \frac{1}{1-\epsilon}\phi(\vsbar),
\]
with probability at least $1-\delta$, where $\phi(\vs) := \vd^\T \mL_{\vs}^\dagger \vd.$ 
\end{theorem}
The proof appears in Appendix~\ref{apdx:thm:bern_rounding_main}.

\subsection{Approximation certificate}
\label{sec:convergence_certification}
In this section, we study an explicit, computable optimality certificate for a family of switch distributions.
    
\begin{theorem}
    \label{thm:fw_gap_convergence}
    Let $\setS \subseteq \s{0,1}^m$ be a convex set of almost surely connected switching distributions, and let $\vs_t \in \setS$ be the $t$-th Frank-Wolfe iterate of the switching probabilities in the Frank-Wolfe algorithm. Set
    \begin{equation}
    \label{eq:lmo_vector}
    \vv_t^\star \in \argmin_{\vv \in S} \ \ip{\nabla\phi(\vs_t)}{\vv},
    \end{equation}
    Then, if
    $
    \ip{\nabla \phi(\vs_t)}{\vs_t - \vv_t^\star} \leq \tau \cdot \phi(\vs_t),
    $
    we have that
    \[
    \phi(\vs_t) - \phi(\vs_\star) \leq \frac{\tau}{1-\tau} \phi(\vs_\star)\; \iff \; \phi(\vs_t) \leq \frac{1}{1-\tau}\phi(\vs_\star),
    \]
    where $\vs_\star$ is a global minimizer of the convex function $\phi(\cdot)$ over $S$.
\end{theorem}
In particular, for the top-$q$ set with backbone,  we have the following result.
\begin{corollary}[$\alpha$-competitive approximation certificate]
    \label{thm:alpha_certificate}
    Let $\vs_0 \in \c{0,1}^m$ be a connected \emph{backbone}, and consider the feasible region of connected configurations  
    \[
    \begin{aligned}
    S_q(\vs_0) := \c{\vs \in \s{0,1}^m  :  
    \norm{\vs}_1 \leq q, \quad s_e =1 \; \forall e \in \mathsf{supp}(\vs_0)}.
    \end{aligned}
    \]
    Fix a parameter $\alpha \in (0,1)$. For any connected configuration $s \in S_q(s_0)$, define the linear minimization oracle $v_\star(s) := \argmin_{v \in S_q(s_0)} \, \ip{\nabla \phi(s)}{v}$, where, for each $e \in E$, 
    \begin{equation}
    \label{eq:specific_lmo_entrywise}
    \p{v_\star(s)}_e := \begin{cases}
        1 & \text{if } e \in \mathsf{supp}(s_0) \\
        1 & \text{if } e \in \mathsf{Top}_{q-n+1}\c{\p{\mId_m-\mS_0}\nabla \phi(s)}\\
        0 & \mathsf{otherwise,}
    \end{cases} 
    \end{equation}
    where $\mS_0 := \diag(\vs_0).$ If, for some $\tau>0$, we have
    \[
    \ip{\nabla \phi(s)}{s-v_\star(s)} \leq \tau  \phi(s), \quad \tau := \frac{\alpha}{1+\alpha}, \quad \alpha \in (0,1), 
    \]
    then $s$ is an $\alpha$-approximate solution to the network reconfiguration problem. That is, for any potentially non-unique minimizer $s_\star \in \argmin_{s\in S_q(s_0)}\, \phi(s)$, we have
    \[
    \phi(s) \leq \p{1+\alpha} \phi(s_\star), 
    \]
    with $\alpha = \frac{\tau}{1-\tau} <1,$ i.e., the approximation is \textit{efficient}.
\end{corollary}

\begin{remark}
    When computing the convergence certificate with the approximated gradient of Lemma \ref{lemma:approximate_gradient}, the Frank–Wolfe gap estimated by the approximate Laplacian solver differs from the true gap by at most a factor $1+O(\epsilon)$. 
\end{remark}

\begin{theorem}[End-to-end approximation guarantee with runtime complexity]
\label{thm:FW-rounded-output-convergence}
    Assume that the fixed backbone~$T\subseteq E$ satisfies~$\norm{\vd}_{\mL_T}^2 \leq 1$. Suppose that $p_{\sf min} = 0$. For any $\delta \in (0,1)$, Algorithm \textsc{RandReconfig} returns a random integral solution $\vstilde \in \c{0,1}^m$ that satisfies, with probability at least $1-\delta$,
    \[
    \phi(\vstilde) \leq \frac{1}{1-\epsilon}\phi(\vs_t) \leq \frac{1+\alpha}{1-\epsilon} \phi(\vs_\star), \quad  \mathsf{with} \quad \epsilon \lesssim \sqrt{\log(n/\delta)}
    \]
    and~$\Expec[]{\norm{\vstilde}}= q$, in total running time $\tilde{O}(mq/\alpha),$ where $\tilde{O}(\cdot)$ suppresses polylogarithmic factors in $n$ and the Laplacian solver tolerance. 
\end{theorem}
The proof appears in Appendix~\ref{apdx:end-to-end-FW-l2}. Appendix~\ref{apdx:satisfying_the_budget_high_prob} gives an overview of a method to augment the probabilities to ensure that the budget constraint is satisfied with high probability in a single pass.

\section{Numerical Results}
\label{sec:num}

\subsection{Baseline formulation}
\label{sec:baseline-form}
The network reconfiguration studied in the present paper can be formulated as a mixed-integer second-order cone program (MISOCP). This structure is compatible with some commercial solvers, like Gurobi, to ostensibly solve the problem efficiently. Introducing a change of variable, we obtain the program
\begin{subequations}
    \label{eq:misocp}
    \begin{align}
        \min_{\vu,\vs,\vf} \quad & 2 \sum_{e \in E} w_e^{-1} u_e\\
        \st \quad &\vu \geq \vzero, \ \vs \in \c{0,1}^m, \ \vf \in \R^m,\\
        \quad &u_e s_e \geq \frac{1}{2}f_e^2, \quad \forall e \in E, \label{eq:misocp_soc}\\
        \quad &s_e =1 \quad \forall e \in S_0,\\
        \quad &\norm{\vs}_1 \leq q, \label{eq:misocp_power}\\
        \quad &\mA^\top \vf = \vd.
    \end{align}
\end{subequations}
The program \eqref{eq:misocp} is a mixed-integer second-order cone program (MISOCP). The objective function is a linearization of the congestion objective $\phi(s)$, where $u_e$ is an upper bound on the congestion of edge $e$. The constraints \eqref{eq:misocp_soc} are  second-order cone constraints (condtional on $s_e$) that enforce $u_e \geq \frac{1}{2}\frac{f_e^2}{s_e}$ when $s_e = 1$ and $u_e \geq 0$ when $s_e = 0$. The constraint \eqref{eq:misocp_power} enforces the power constraint. Finally, the constraint $\mA^\top \vf = \vd$ enforces flow conservation.

\subsection{Numerical implementation and test of the approximation scheme}
\label{sec:implementation}

\begin{figure*}[t!]
    \centering
    \includegraphics[width=0.975\linewidth,keepaspectratio]{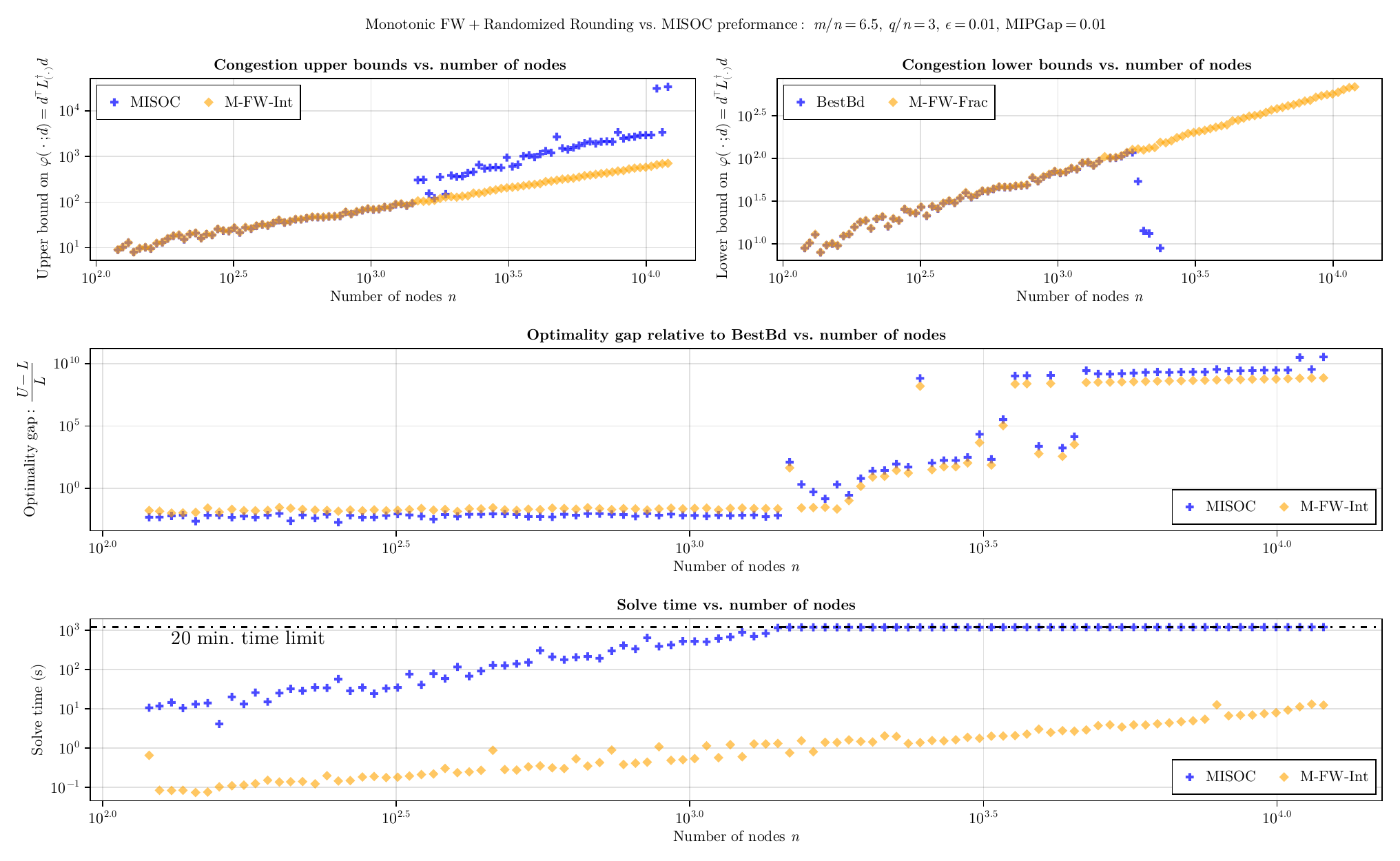}
    \caption{Comparison of Alg. \ref{alg:switching} and \eqref{eq:misocp} using Gurobi's MISOCP solver vs. number of nodes in a test graph. The bottom sub-figure depicts an orders-of-magnitude speedup.} 
    \label{fig:results}
\end{figure*}

Algorithm \ref{alg:gradient_compute} and Algorithm \ref{alg:switching} were implemented in the Julia programming language, leveraging the fast Laplacian solver available in \texttt{Laplacians.jl}, presented in \cite{gao_robust_2023}. We generated random graphs over a large range of network sizes, using a different seed for each size. The experiments were run on a conventional laptop with a Ryzen 4750U chipset.

Similarly, as a baseline comparison, we implement the MISOCP formulation \eqref{eq:misocp} in the JuMP algebraic modeling language \cite{dunning_jump_2017}, which was then fed into Gurobi's MISOCP solver. As shown in Fig.~\ref{fig:results}, our implementation of Algorithm~\ref{alg:gradient_compute}, based on the monotonic Frank-Wolfe algorithm with randomized rounding, yields up to a 1000$\times$ improvement in the total solve time from Gurobi. In particular, the upper and lower bounds on the congestion are close across both methods. Moreover, when Gurobi fails to declare convergence within a  stipulated 20 minute time limit, the FW procedure achieves a superior optimality gap relative to the best incumbent lower bound determined by Gurobi.

\section{Conclusion}
\label{sec:conclusion}

We presented an approximation scheme to efficiently solve network reconfiguration problems. We proposed a particular instance of such an algorithm, \textsc{RandReconfig}, based on a conditional gradient procedure with randomized rounding. We demonstrated that our method outperforms a commercial MINLP solver by orders of magnitude on a relevant problem. Inclusion of line flow limits (edge capacities) is an essential component for future work, in addition to allowing some entries of the demand vector to be continuous decision variables (i.e., controllable generation). 


\section*{Acknowledgments}
The authors thank Cameron Khanpour, Sergio A. Dorado-Rojas, Daniel Turizo, and Ashwin Pananjady for insightful discussions. This material is based upon work supported by the National Science Foundation Graduate Research Fellowship under Grant No. DGE-2039655. Any opinion, findings, and conclusions or recommendations expressed in this material are those of the authors and do not necessarily reflect the views of the National Science Foundation.

\bibliography{refs}

\begin{thebibliography}{10}

\bibitem{allen-zhu_near-optimal_2021}
{\sc Z.~Allen-Zhu, Y.~Li, A.~Singh, and Y.~Wang}, {\em Near-optimal discrete optimization for experimental design: a regret minimization approach}, Mathematical Programming, 186 (2021), pp.~439--478.

\bibitem{anderson2014efficientalgorithmunweightedspectral}
{\sc D.~G. Anderson, M.~Gu, and C.~Melgaard}, {\em An efficient algorithm for unweighted spectral graph sparsification}, 2014.

\bibitem{avrom_subset_2013}
{\sc H.~Avron and C.~Boutsidis}, {\em Faster subset selection for matrices and applications}, SIAM Journal on Matrix Analysis and Applications, 34 (2013), pp.~1464--1499.

\bibitem{Babaeinejadsarookolaee_congestion_2023}
{\sc S.~Babaeinejadsarookolaee, B.~Park, B.~Lesieutre, and C.~L. DeMarco}, {\em Transmission congestion management via node-breaker topology control}, IEEE Systems Journal, 17 (2023), pp.~3413--3424.

\bibitem{boumal_concentration_2014}
{\sc N.~Boumal and X.~Cheng}, {\em Concentration of the {Kirchhoff} index for {Erd{\H{o}}s}–{R{\'e}nyi} graphs}, Systems \& Control Letters, 74 (2014), pp.~74--80.

\bibitem{braun_conditional_2025}
{\sc G.~Braun, A.~Carderera, C.~W. Combettes, H.~Hassani, A.~Karbasi, A.~Mokhtari, and S.~Pokutta}, {\em Conditional {Gradient} {Methods}}, Aug. 2025.
\newblock arXiv:2211.14103 [math].

\bibitem{brown_fast_2024}
{\sc A.~Brown, A.~Laddha, and M.~Singh}, {\em Fast algorithms for maximizing the minimum eigenvalue in fixed dimension}, Operations Research Letters, 57 (2024), p.~107186.

\bibitem{carderera2024scalable}
{\sc A.~Carderera, M.~Besancon, and S.~Pokutta}, {\em Scalable {F}rank-{W}olfe on generalized self-concordant functions via simple steps}, SIAM Journal on Optimization, 34 (2024), pp.~2231--2258.

\bibitem{charalambides_matmul_sparselap_2023}
{\sc N.~Charalambides and A.~O. Hero}, {\em Graph sparsification by approximate matrix multiplication}, in 2023 IEEE Statistical Signal Processing Workshop (SSP), 2023, pp.~180--184.

\bibitem{chen_almostlinear_maxflow_2022}
{\sc L.~Chen, R.~Kyng, Y.~P. Liu, R.~Peng, M.~P. Gutenberg, and S.~Sachdeva}, {\em Maximum flow and minimum-cost flow in almost-linear time}, in 2022 IEEE 63rd Annual Symposium on Foundations of Computer Science (FOCS), 2022, pp.~612--623.

\bibitem{chen_spectral_2021}
{\sc Y.~Chen, Y.~Chi, J.~Fan, and C.~Ma}, {\em Spectral methods for data science: {A} statistical perspective}, Foundations and Trends in Machine Learning, 14 (2021), pp.~566--806.
\newblock Publisher: Now Publishers, Inc.

\bibitem{deza_polytopal_2021}
{\sc A.~Deza, J.-B. Hiriart-Urruty, and L.~Pournin}, {\em Polytopal balls arising in optimization}, Contributions to Discrete Mathematics, 16 (2021), pp.~125--138.

\bibitem{dunning_jump_2017}
{\sc I.~Dunning, J.~Huchette, and M.~Lubin}, {\em {JuMP}: A modeling language for mathematical optimization}, SIAM Review, 59 (2017), pp.~295--320.

\bibitem{dvurechensky_generalized_2023}
{\sc P.~Dvurechensky, K.~Safin, S.~Shtern, and M.~Staudigl}, {\em Generalized self-concordant analysis of {Frank}–{Wolfe} algorithms}, Mathematical Programming, 198 (2023), pp.~255--323.

\bibitem{foster1949average}
{\sc R.~M. Foster}, {\em The average impedance of an electrical network}, Contributions to Applied Mechanics (Reissner Anniversary Volume), 333 (1949).

\bibitem{FUKUSHIMA1984169}
{\sc M.~Fukushima}, {\em A modified {F}rank-{W}olfe algorithm for solving the traffic assignment problem}, Transportation Research Part B: Methodological, 18 (1984), pp.~169--177.

\bibitem{gao_robust_2023}
{\sc Y.~Gao, R.~Kyng, and D.~A. Spielman}, {\em Robust and practical solution of {L}aplacian equations by approximate elimination}, June 2023.
\newblock arXiv:2303.00709 [math].

\bibitem{ghosh_growing_2006}
{\sc A.~Ghosh and S.~Boyd}, {\em Growing well-connected graphs}, in 45th {IEEE} {Conference} on {Decision} and {Control}, Dec. 2006, pp.~6605--6611.

\bibitem{ghosh_minimizing_2008}
{\sc A.~Ghosh, S.~Boyd, and A.~Saberi}, {\em Minimizing effective resistance of a graph}, SIAM Review, 50 (2008), pp.~37--66.

\bibitem{haider_siphyR_2025}
{\sc R.~Haider, A.~Annaswamy, B.~Dey, and A.~Chakraborty}, {\em {SiPhyR}: An end-to-end learning-based optimization framework for dynamic grid reconfiguration}, IEEE Transactions on Smart Grid, 16 (2025), pp.~1248--1260.

\bibitem{pmlr-v23-hsu12}
{\sc D.~Hsu, S.~M. Kakade, and T.~Zhang}, {\em Random design analysis of ridge regression}, in Proceedings of the 25th Annual Conference on Learning Theory, vol.~23 of Proceedings of Machine Learning Research, Edinburgh, Scotland, 25--27 Jun 2012, PMLR, pp.~9.1--9.24.

\bibitem{Kooij_2023}
{\sc R.~E. Kooij and M.~A. Achterberg}, {\em Minimizing the effective graph resistance by adding links is {NP}-hard}, Operations Research Letters, 51 (2023), p.~601–604.

\bibitem{kyng_dissertation}
{\sc R.~Kyng}, {\em Approximate {G}aussian elimination}.
\newblock Ph.D. dissertation, Yale University, 2017.

\bibitem{kyng_approx_gaussian_2016}
{\sc R.~Kyng and S.~Sachdeva}, {\em Approximate {G}aussian elimination for {L}aplacians - {F}ast, sparse, and simple}, in 2016 IEEE 57th Annual Symposium on Foundations of Computer Science (FOCS), 2016, pp.~573--582.

\bibitem{kyng2016approximategaussianeliminationlaplacians}
{\sc R.~Kyng and S.~Sachdeva}, {\em Approximate {G}aussian elimination for {L}aplacians: Fast, sparse, and simple}, 2016.

\bibitem{lau_spectral_2020}
{\sc L.~C. Lau and H.~Zhou}, {\em A spectral approach to network design}, in 52nd Annual ACM SIGACT Symposium on Theory of Computing, STOC 2020, 2020, p.~826–839.

\bibitem{lehmann2014complexitydcswitchingproblems}
{\sc K.~Lehmann, A.~Grastien, and P.~V. Hentenryck}, {\em The complexity of {DC}-switching problems}, 2014.

\bibitem{li_maximizing_trees_2020}
{\sc H.~Li, S.~Patterson, Y.~Yi, and Z.~Zhang}, {\em Maximizing the number of spanning trees in a connected graph}, IEEE Transactions on Information Theory, 66 (2020), pp.~1248--1260.

\bibitem{moondra2024promisegraphsparsificationdecomposition}
{\sc J.~Moondra, P.~C. Lotshaw, G.~Mohler, and S.~Gupta}, {\em Promise of graph sparsification and decomposition for noise reduction in qaoa: Analysis for trapped-ion compilations}, 2024.

\bibitem{LOpt_ECC2015}
{\sc H.~Nagarajan, S.~Rathinam, and S.~Darbha}, {\em On maximizing algebraic connectivity of networks for various engineering applications}, in European Control Conference (ECC), 2015, pp.~1626--1632.

\bibitem{oliveira_concentration_2010}
{\sc R.~I. Oliveira}, {\em Concentration of the adjacency matrix and of the {Laplacian} in random graphs with independent edges}, Feb. 2010.
\newblock arXiv:0911.0600 [math].

\bibitem{oliveira2013lowertailrandomquadratic}
\leavevmode\vrule height 2pt depth -1.6pt width 23pt, {\em The lower tail of random quadratic forms with applications to ordinary least squares}, Probability Theory and Related Fields, 166 (2016), pp.~1175--1194.

\bibitem{pmlr-v99-ostrovskii19a}
{\sc D.~M. Ostrovskii and A.~Rudi}, {\em Affine invariant covariance estimation for heavy-tailed distributions}, in Proceedings of the Thirty-Second Conference on Learning Theory, vol.~99 of Proceedings of Machine Learning Research, PMLR, 25--28 Jun 2019, pp.~2531--2550.

\bibitem{pham_scalable_fw_2023}
{\sc C.~B. Pham, W.~Griggs, and J.~Saunderson}, {\em A scalable {F}rank-{W}olfe-based algorithm for the max-cut {SDP}}, in 40th International Conference on Machine Learning, vol.~202 of Proceedings of Machine Learning Research, PMLR, 23--29 Jul 2023, pp.~27822--27839.

\bibitem{pokutta_frank-wolfe_2024}
{\sc S.~Pokutta}, {\em The {Frank}-{Wolfe} {Algorithm}: {A} {Short} {Introduction}}, Jahresbericht der Deutschen Mathematiker-Vereinigung, 126 (2024), pp.~3--35.

\bibitem{raghavan_randomized_1987}
{\sc P.~Raghavan and C.~D. Tompson}, {\em Randomized rounding: {A} technique for provably good algorithms and algorithmic proofs}, Combinatorica, 7 (1987), pp.~365--374.

\bibitem{ru_maximizing_2025}
{\sc X.~Ru, W.~Xia, and M.~Cao}, {\em Maximizing the smallest eigenvalue of grounded {Laplacian} matrices via edge addition}, Automatica, 176 (2025), p.~112238.

\bibitem{pmlr-v51-sadhanala16}
{\sc V.~Sadhanala, Y.-X. Wang, and R.~Tibshirani}, {\em Graph sparsification approaches for {L}aplacian smoothing}, in 19th International Conference on Artificial Intelligence and Statistics, vol.~51 of Proceedings of Machine Learning Research, Cadiz, Spain, 09--11 May 2016, PMLR, pp.~1250--1259.

\bibitem{shannon_concavity_1956}
{\sc C.~E. Shannon and D.~W. Hagelbarger}, {\em Concavity of {Resistance} {Functions}}, Journal of Applied Physics, 27 (1956), pp.~42--43.

\bibitem{LOpt_TCNS2024}
{\sc N.~Somisetty, H.~Nagarajan, and S.~Darbha}, {\em Optimal robust network design: Formulations and algorithms for maximizing algebraic connectivity}, IEEE Transactions on Control of Network Systems, 12 (2024), pp.~918--929.

\bibitem{somisetty_spectral_robots_2024}
\leavevmode\vrule height 2pt depth -1.6pt width 23pt, {\em Spectral graph theoretic methods for enhancing network robustness in robot localization}, in 63rd IEEE Conference on Decision and Control (CDC), 2024, pp.~8811--8818.

\bibitem{spielman_graph_2011}
{\sc D.~A. Spielman and N.~Srivastava}, {\em Graph sparsification by effective resistances}, SIAM Journal on Computing, 40 (2011), pp.~1913--1926.

\bibitem{spielman_teng_2004}
{\sc D.~A. Spielman and S.-H. Teng}, {\em Nearly-linear time algorithms for graph partitioning, graph sparsification, and solving linear systems}, in 36th Annual ACM Symposium on Theory of Computing, STOC '04, 2004, p.~81–90.

\bibitem{spielman_spectral_2011}
\leavevmode\vrule height 2pt depth -1.6pt width 23pt, {\em Spectral sparsification of graphs}, SIAM Journal on Computing, 40 (2011), pp.~981--1025.

\bibitem{spielman_nearly-linear_2012}
\leavevmode\vrule height 2pt depth -1.6pt width 23pt, {\em Nearly-linear time algorithms for preconditioning and solving symmetric, diagonally dominant linear systems}, Sept. 2012.
\newblock arXiv:cs/0607105.

\bibitem{stott_dc_revisited_2009}
{\sc B.~Stott, J.~Jardim, and O.~Alsac}, {\em {DC} power flow revisited}, IEEE Transactions on Power Systems, 24 (2009), pp.~1290--1300.

\bibitem{sun2019generalized}
{\sc T.~T. Sun and Q.~Tran-Dinh}, {\em Generalized self-concordant functions: a recipe for {N}ewton-type methods}, Mathematical Programming, 178 (2019), pp.~145--213.

\bibitem{teng_laplacian_2010}
{\sc S.-H. Teng}, {\em The {L}aplacian paradigm: Emerging algorithms for massive graphs}, in Theory and {Applications} of {Models} of {Computation}, Springer, 2010, pp.~2--14.

\bibitem{tropp2011freedmansinequalitymatrixmartingales}
{\sc J.~Tropp}, {\em Freedman's inequality for matrix martingales}, Electronic Communications in Probability, 16 (2011), pp.~262--270.

\bibitem{Tropp_2011}
{\sc J.~A. Tropp}, {\em User-friendly tail bounds for sums of random matrices}, Foundations of Computational Mathematics, 12 (2011), p.~389–434.

\bibitem{vandersar2025optimizingpowergridtopologies}
{\sc E.~van~der Sar, A.~Zocca, and S.~Bhulai}, {\em Optimizing power grid topologies with reinforcement learning: A survey of methods and challenges}, 2025.

\bibitem{wang_improving_2014}
{\sc X.~Wang, E.~Pournaras, R.~E. Kooij, and P.~Van~Mieghem}, {\em Improving robustness of complex networks via the effective graph resistance}, The European Physical Journal B, 87 (2014), p.~221.

\bibitem{yurtsever_scalable_2021}
{\sc A.~Yurtsever, J.~A. Tropp, O.~Fercoq, M.~Udell, and V.~Cevher}, {\em Scalable {Semidefinite} {Programming}}, SIAM Journal on Mathematics of Data Science, 3 (2021), pp.~171--200.

\bibitem{zhou_maximizing_2025}
{\sc X.~Zhou, R.~Wang, W.~Li, and Z.~Zhang}, {\em Maximizing the smallest eigenvalue of grounded {Laplacian} matrix}, Journal of Global Optimization, 91 (2025), pp.~807--828.

\bibitem{zhou_efficient_2025}
{\sc X.~Zhou, A.~N. Zehmakan, and Z.~Zhang}, {\em Efficient algorithms for minimizing the {K}irchhoff index via adding edges}, IEEE Transactions on Knowledge and Data Engineering, 37 (2025), pp.~3342--3355.

\end{thebibliography}






\clearpage
\appendix

\section{Proofs}

\subsection{Proof of Lemma \ref{lem:ghosh_gradient}}
The following Lemma is a refreshed version of the result in \cite{ghosh_minimizing_2008}. 
\ifapdxthms
\begin{lemma}
    \label{lem:ghosh_gradient}
    Let~$R(\vs) = \sum_{i<j} \varrho_{ij}(\vs) = \frac{1}{2} \vone^\T \mR_{\vs} \vone$ be the total effective resistance of the graph~$G_{\vs} = (N,E,\vw \circ \vs)$. Then, 
    \[
    \frac{\partial}{\partial s_{ij}} R(\vs) = -n w_{ij} \ve_{ij}^\T\mL^{2\dagger}_{\vs} \ve_{ij}, \qquad ij \in E.
    \]
\end{lemma}
\fi
\begin{proof}[Proof of Lemma \ref{lem:ghosh_gradient}]
    In the setting of \cite{ghosh_minimizing_2008}, the switching variables are fixed as~$\vs = \vone$, and the gradient of the total effective resistance with respect to the \textit{weights} is given entrywise as the quadratic forms 
    \[
    \p{\nabla_{\vw} R}_{ij} = \frac{\partial }{\partial w_{ij}}R = -n \ve_{ij}^\T\mL^{2\dagger}_{\vs} \ve_{ij}, \qquad ij \sim 1,\ldots,m,
    \]
    where $\mL^{2\dagger}_{\vs} := \mL^\dagger_{\vs} \mL^\dagger_{\vs}$ is the square of the Moore-Penrose pseudo-inverse of the Laplacian matrix $\mL_{\vs}$. It is well-known that the total resistance $R(\vs)$ is a homogeneous function of degree $-1$;\footnote{Note: This also means that the gradient is homogeneous of degree $-2$.} this means that
    \[
    \ip{\nabla_{\vs}R(\vs)}{\vs} = - R(\vs).
    \]
    
\end{proof}


\subsection{Proof of Lemma \ref{lem:switching_gradient}}
\ifapdxthms
\begin{lemma}[Congestion gradient]
\label{lem:switching_gradient}
    For a fixed demand $\vd \perp \vone$, the gradient of the  relaxed congestion \eqref{eq:relaxed_cong} is given elementwise for each edge $e \in E$ as
    \begin{equation}
        \label{eq:congestion_gradient_entries}
        \frac{\partial}{\partial s_e}\,\phi(\vs) = - w_e \abs{\ip{\va_e}{\mL_{\vs}^\dagger \vd}}^2 := -w_e\Delta_e^2 \leq 0, 
    \end{equation}
    where $\Delta_e = \hat{x}_i - \hat{x}_j$ is the voltage difference \textit{solution map} across edge $e=(i,j) \in E$. 
\end{lemma}
\fi
 \begin{proof}[Proof of Lemma \ref{lem:switching_gradient}]
    Relax the switching decisions as~$\vs \in \s{0,1}^m$, and consider specified demands~$\vd \perp \vone$. Recall that we define the congestion or energy as~$\phi : \c{0,1}^m \times \R_+$ as
    \begin{equation}
        \label{eq:random_congestion}
        \phi(\vs) = \vd^\T \mL_{\vs}^\dagger \vd = \trace\p{\mL_{\vs}^\dagger \vd \vd^\T}. 
    \end{equation}
    The gradient of the electrical congestion \eqref{eq:random_congestion} with respect to the relaxed switching variables $\vs$ can be computed elementwise for each potential edge $\ell \sim (i,j)$ as follows:
    \begin{align*}
        \p{\nabla_{\vs} \phi}_{\ell} &= \frac{\partial}{\partial s_{\ell}} \trace\p{\mL_{\vs}^\dagger \vd \vd^\T} \\
        &= \frac{\partial}{\partial s_{\ell}}  \trace\s{\p{\p{\mL_{\vs} +  \vone \vone^\T/n}^{-1} - \vone\vone^\T/n} \vd \vd^\T} \\
        &\overset{(1)}{=} \trace\s{ \frac{\partial}{\partial s_\ell} \p{\mL_{\vs} +  \vone \vone^\T/n}^{-1}  \vd \vd^\T} \\
        &\overset{(2)}{=} -\trace\s{\p{\mL_{\vs} + \vone \vone^\T/n}^{-1} \ve_{\ell}\ve_{\ell}^\T w_\ell \p{\mL_{\vs} + \vone\vone^\T/n}^{-1}\vd\vd^\T} \\
        &\overset{(3)}{=} -w_{\ell} \vd^\T \mL_{\vs}^\dagger \mE_{\ell} \mL_{\vs}^\dagger \vd.
    \end{align*}
    In the above display, step (1) is due to the fact that $\vd \perp \vone$, step (2) is due to the matrix derivative of the inverse, and step (3) is due to cyclic property of the trace, and the fact that $\ve_{\ell} \in \Null(\vone \vone^\T)$, so 
    $$
    \mL_{\vs}^\dagger\ve_{\ell} = \p{\p{\mL_{\vs} + \vone \vone^\T/n}^{-1} - \vone\vone^\T/n}\ve_{\ell}.
    $$
    
    Finally, let $\hat\vx(\vs,\vd)= \mL_{\vs}^\dagger \vd $ be the voltages computed by solving the Laplacian system induced by switching strategy $\vs$ given demands $\vd$. Then, for each edge $\ell \sim ij \in E$, we have
    \[
    \p{\nabla_{\vs} \phi}_{ij} = - w_{ij} \abs{\ip{\hat\vx(\vs,\vd)}{\ve_{ij}}}^2 = - w_{ij} \abs{\hat{x}_i(\vs,\vd) - \hat{x}_j(\vs,\vd)}^2 = -w_{ij}\Delta_{ij}^2(\vs,\vd),
    \]
    where $\Delta_{ij}(\vs,\vd) = \hat{x}_i(\vs,\vd) - \hat{x}_j(\vs,\vd)$ is the voltage difference across edge $ij \in E$ connecting nodes $i<j$ induced by the demands $\vd$ on the switched graph $G_{\vs} = (N,E,\vw \circ \vs)$. Equivalently, by applying the incidence matrix $\mA$, the vector of voltage differences $\Delta(\vs,\vd) := \mA \hat\vx(\vs,\vd)$.
\end{proof}

\subsection{Remark: Closing all switches is the unconstrained global minimizer}
\begin{remark}
    Since $\mL_{\vs}^\dagger \succeq 0$ for all $\vs$, and and $w_{\ell}\geq 0$ for all $\ell$, it follows that the symmetric matrix product  $\mL_{\vs}^\dagger \mE_{\ell} \mL_{\vs}^\dagger \succeq 0$. Consequently, for all $\vd \perp \vone$, we have 
    \[
    \nabla_{\vs} \phi(\vs) \leq \vzero \quad \textsf{for all} \quad \vs \in \R^m_+;
    \]
    and thus, 
    \[
    \phi(\vone,\vd) \leq \phi(\vs) \quad \textsf{for all } \vs \in \s{0,1}^m \setminus\c{\vone},
    \]
    i.e., the all-ones vector~$\vone$ is the (unconstrained) global minimizer of the random congestion objective function $\phi(\cdot,\vd)$ for any fixed $\vd \perp \vone$:
    \[
    \min_{\vs \in \s{0,1}^m} \ \phi(\vs) = \phi(\vone,\vd) \quad \textsf{for all} \quad \vd \perp \vone.
    \]
\end{remark}

\subsection{Proof of Lemma \ref{lemma:voltage_cauchy_schwarz}}
\ifapdxthms
\begin{lemma}[Cauchy–Schwarz for voltage differences]
        \label{lemma:voltage_cauchy_schwarz}
        Let $\vd \in \R^n$ be a vector of demands, and let $ \vx = \mL_{\vs}^\dagger \vd$ be the voltages induced by the demands under switching strategy $\vs \in \s{0,1}^m$. Then, for any edge $ij \in E$ and switching strategy $\vs$, we have
        \begin{equation}
            \label{eq:voltage_cauchy_schwarz}
            \Delta_{ij}^2 = \abs{\va_{ij}^\T \vx}^2  \leq \rho_{ij} \cdot \phi(\vs) \qquad \text{for all } (i,j) \in E.
        \end{equation}
    \end{lemma}
\fi
\begin{proof}[Proof of Lemma \ref{lemma:voltage_cauchy_schwarz}]
    For every edge $ij$, 
    \begin{align*}
        \abs{\va_{ij}^\T \vx}^2 &= \abs{\va_{\ell}^\T \mL_{\vs}^\dagger \vd}^2\\
        &= \abs{\ip{\mL_{\vs}^{\dagger/2} \va_\ell}{\mL_{\vs}^{\dagger/2} \vd}}^2 \tag{$\mL_{\vs}^\dagger \succeq 0$}\\
        &\leq \norm{\mL_{\vs}^{\dagger/2} \va_\ell}_2^2 \norm{\mL_{\vs}^{\dagger/2} \vd}_2^2 \qquad \tag{Cauchy-Schwarz}\\
        &= \rho_\ell \cdot \phi(\vs).  \tag{definition of $\phi$ and $\rho$}
    \end{align*}
    This completes the proof.
\end{proof}

\subsection{Proof of Lemma \ref{lemma:hessian_operator_norm_bound}}
\label{sec:lemma:hessian_operator_norm_bound}
\ifapdxthms
 \begin{lemma}[Operator norm bound on the Hessian]
    \label{lemma:hessian_operator_norm_bound}
    Given a switching strategy set $\setS \subseteq \s{0,1}^m$ and demands~$\vd$, let $\mH(\cdot) : \setS \to \S^m_+$ be the Hessian of the congestion function $\phi(\cdot)$. The operator norm of the Hessian $\mH(\cdot)$ is bounded as follows:
    \begin{align*}
    \opnorm{\mH} &\leq L := 2\max_{\vs \in \setS} \phi(\vs) =  \frac{2 \norm{\vd}_2^2}{\underset{\vs \in \setS}{\inf}\lambda_2(\mL_{\vs})} \leq \frac{2 \norm{\vd}_2^2}{\lambda_2\p{\mL_{\vs_0}}}.
    \end{align*}
    Moreover, if $\norm{\vd}_2 \leq 1$ always, and the template is such that $\lambda_2(\mL_{\vs_0}) \geq 1$, then it holds that
    $ \opnorm{\mH} \leq 2$. 
    
\end{lemma}
\fi
\begin{proof}[Proof of Lemma \ref{lemma:hessian_operator_norm_bound}]
        Fix demands $\vd$, and set $\vx := \mL^\dagger(\vs)\vd$, $\Delta := \mA \vx$. Recall that $\frac{\partial\phi}{\partial s_\ell}(\vs,\vd) = -w_\ell (\va_\ell^\T \vx)^2 = -w_\ell \Delta_\ell^2$. By direct calculation, we obtain the entries of the Hessian of the congestion function $\phi(\vs)$ with respect to the switching variables $\vs$ as follows:
       \begin{subequations}
        \label{eq:hessian_congestion_elementwise}
        \begin{align}
        \s{\mH}_{\ell,\ell'} = \frac{\partial^2 \phi}{\partial s_\ell \partial s_{\ell'}}(\vs,\vd) &= 
        \begin{cases}
            2w_\ell^2 \p{\va_\ell^\T \vx}^2 \p{\va_{\ell}^\T\mL^\dagger\va_\ell}, & \text{if } \ell = \ell'\\
            2w_\ell w_{\ell'} \p{\va_\ell^\T \vx} \p{\va_{\ell'}^\T \vx} \p{\va_\ell^\T\mL^\dagger\va_{\ell'}} & \text{if } \ell \neq \ell',\\
        \end{cases}\\
        &=\begin{cases}
            2\Delta_\ell^2 w_\ell^2 \rho_\ell & \text{if } \ell = \ell'\\
            2\Delta_\ell \Delta_{\ell'} w_\ell w_{\ell'} \rho_{\ell'} & \text{if } \ell \neq \ell'.
        \end{cases}
        \end{align}
       \end{subequations}
       In particular, the diagonal entries of the Hessian can be written as, for each edge $e \in E$, 
       \[
       \s{\mH}_{e,e} = -2 \underbrace{w_{e} \rho_{e}}_{=\ell_{e}(\vs)}  \p{\nabla \phi}_{e} = -2 \ell_e \p{\nabla \phi}_e.
       \]
        Define the projection matrix  as $\mP := \mW^{1/2}\mA \mL^\dagger \mA^\T \mW^{1/2}$. The matrix $\mP$ is the orthogonal projector onto $\Range\p{\mW^{1/2} \mA}$. It is known (cf. \cite{spielman_graph_2011}) that the maximum eigenvalues $\opnorm{\mP} = 1$ with multiplicity $n-1$. Define a vector $\vzeta \in \R^m$ of scaled voltage perturbations
        \[
       \zeta_{e} := \sqrt{w_{e}} \va_{e}^\T \vx := \sqrt{w_{e}} \Delta_{e}, \quad e \sim 1,\ldots,m.
        \]
        Then, the Hessian can be expressed as follows:
        \[
        \mH(\vs) = 2\diag(\vzeta) \mP \diag(\vzeta).
        \]
         As a consequence of Lemma \ref{lemma:voltage_cauchy_schwarz},
        \[
        w_{\ell} \Delta_\ell^2 \leq w_\ell \rho_\ell \cdot \phi(\vs) \leq \phi(\vs), \quad \ell \sim 1,\ldots,m,
        \]
        which holds because Rayleigh's monotonicity principle gives us $w_\ell \rho_\ell \leq 1$. Consequently, we have
        \[
        \opnorm{\diag(\vzeta)}^2 = \max_\ell w_{\ell} \Delta_\ell^2 \leq \phi(\vs).
        \]

        Therefore, the operator norm of the Hessian is bounded as
        \begin{align*}
           \sup_{\vs \in \setS}\ \opnorm{\mH(\vs)} &=  \sup_{\vs \in \setS}\ \opnorm{2\diag(\vzeta) \mP \diag(\vzeta)}\\
           &\overset{(1)}{\leq} 2 \cdot \sup_{\vs \in \setS}\ \opnorm{\diag(\vzeta)}^2 \cdot \opnorm{\mP} \\
           &\overset{(2)}{\leq} 2 \cdot \sup_{\vs \in \setS}\ \phi(\vs)\\
           &=2\sup_{\vs \in \setS} \norm{\vd}_{\mL_{\vs}}^2\\
           &\overset{(3)}{\leq} \frac{2 \norm{\vd}_2^2}{\inf_{\vs \in \setS} \lambda_2(\mL_{\vs})}\\
           &\overset{(4)}{\leq}  \frac{2 \norm{\vd}_2^2}{\lambda_2(\mL_{\vs_0})}=:L,
        \end{align*}
        where step (1) follows from the fact that $\opnorm{\mP} = 1$ with multiplicity $n-1$, step (2) follows from Lemma \ref{lemma:voltage_cauchy_schwarz}, and step (3) follows from the fact that $\phi(\vs) \leq \frac{1}{\lambda_2(\mL_{\vs})} \norm{\vd}_2^2$. Lastly, step (4) is due to the fact that $1/\lambda_2(\mL_{\vs})$ is montonically increasing as $\vs \to \vzero$; hence, the graph corresponding to~$\vs_0$; that is, the graph with the least number of edges in~$\setS$, attains the maximal value of~$1/\lambda_2(\cdot).$  Finally, under certain connectivity assumptions, we can maintain $\lambda_2 \geq 1$. 
        Further, by assumption that $\norm{\vd}_2^2 \leq 1$, we have that $L=2$.
    \end{proof}

    \subsection{Proof of Lemma \ref{lemma:homogeneity}}
    \label{apdx:lemma:homogeneity}
    \begin{proof}[Proof of Lemma \ref{lemma:homogeneity}]
         Reminiscient of \cite{ghosh_minimizing_2008}, note that the congestion $\phi(\cdot)$ can be rewritten as:
    \[
    \phi(\vs) = \vd^\T \mL_{\vs}^+ \vd =\vx(\vs)^\T\mL_{\vs} \vx(\vs) = \sum_{e \in E} s_e w_e \abs{\ip{\va_e}{\vx(\vs)}}^2 = -\sum_{e \in E} s_e \cdot \frac{\partial }{\partial s_e}\phi(\vs) =  -\ip{\nabla \phi(\vs)}{\vs}
    \]
    for any $\vs \in \setD_T$. Thus, we have shown that $\phi$ is a homogeneous function of degree $-1$, as desired. What remains is to show convexity. As remarked by \cite{ghosh_minimizing_2008}, a highly related result, shown in \cite{shannon_concavity_1956}, is that the effective resistance between two electrical points $i$ and $j$:
    \[
    \rho_{ij}(\vs) = \p{\ve_i -  \ve_j}^\T \mL_{\vs}^\dagger \p{\ve_i - \ve_j}, 
    \]
    is a concave function in $r_{ij} := 1/(s_{ij}w_{ij})$ for $w_{ij},s_{ij}\geq0$, but this is not sufficient to imply convexity in $\phi$; moreover, the technique of \cite{ghosh_minimizing_2008} does not yield a straightforward path. We will take a different path to show that $\phi$ is \textit{convex} in $\vs$,  but not strictly so.
    
    
    To establish this, recall that~$\trace \mX^{-1}$ is a convex function over positive semi-definite matrices~$\mX= \mX^\T \succeq \vzero$. Then, recall that the pseudoinverse of the switched Laplacian operator is
    \[
    \mL_{\vs}^\dagger = \p{\mL_{\vs} + \vone\vone^\T/n}^{-1} = \p{\mA^\T \mW^{1/2} \mS \mW^{1/2} \mA + \vone \vone^\T /n}^{-1}.
    \]
    Thus, because the affine mapping~$\vs \mapsto \mL_{\vs} + \vone\vone^\T/n$ is one-to-one, the function~$\trace\p{\mL_{\vs}^\dagger}$ is convex over~$\vs \in \s{0,1}^m$. Consequently, we can show that the congestion
    \[
     \phi(\vs) = \vd^\T \mL_{\vs}^\dagger \vd = \trace\p{\vd \vd^\T \mL_{\vs}^\dagger} = \trace\p{\vd\vd^\T\p{\mL_{\vs} + \vone\vone^\T/n}^{-1}}
    \]
    is convex in $\vs$ for $\vd\perp \vone$. 
    
    First, when $\mM = \vd\vd^\T$, we have the variational identity
    \begin{equation}
    \label{eq:variational_congestion}
    \phi(\vs) = \vd^\T \mL_{\vs}^\dagger\vd = \sup_{\vz \in \R^n}\; 2 \vd^\T \vz - \vz^\T \mL_{\vs} \vz,
    \end{equation}
    which holds because, by the non-negativity of the norm, 
    $$
    \vz^\T \mL_{\vs} \vz - 2 \vd^\T \vz + \vd^\T \mL_{\vs}^\dagger \vd = \p{\vz - \mL_{\vs}^\dagger \vd}^\T \mL_{\vs} \p{\vz  - \mL_{\vs}^\dagger \vd} = \norm{\vz - \mL_{\vs}^\dagger \vd}_{\mL_{\vs}} \geq 0,
    $$
    with equality at $\vz = \mL_{\vs}^\dagger \vd$, which is the unique maximizer of \eqref{eq:variational_congestion}.
    Since for each fixed $\vz$, the mapping $\mL_{\vs} \mapsto - \vz^\T \mL_{\vs} \vz$ is affine in $\mL_{\vs}$, and indeed, affine in $\vs$, we have that $\phi(\vs)$ is a pointwise supremum of a family of affine functions of $\vs$, and hence, convex.
    \end{proof}

    \begin{remark}
        A stronger statement can be proved\textemdash the function $\trace\p{\mM \mL_{\vs}^\dagger}$ is strictly convex over switching probabilities $\vs \in \s{0,1}^m$ if and only if $\mM \succ \vzero$. 
    \end{remark}

    \subsection{Proof of Lemma \ref{lemma:approximate_gradient}}
    \label{apdx:approx:grad}
    \begin{proof}[Proof of Lemma \ref{lemma:approximate_gradient}]
    A single call to $\mathsf{Solve}(\mL_{\vs},\vd,\delta,\varepsilon)$ yields $\hat{\vx}$ in at most~$O(m\log^c n\,\log(1/\varepsilon))$ time; by Lemma~\ref{lem:switching_gradient}, each gradient entry is $-w_e(\va_e^\top\hat{\vx})^2$, so one linear pass over edges suffices.
    
    Set $\vx=\mL_{\vs}^\dagger\vd$ and $\vz:=\hat{\vx}-\vx$. The solver guarantee gives, with probability $\ge 1-O(\delta)$,
    $\|\vz\|_{\mL_{\vs}}\le \varepsilon\|\vx\|_{\mL_{\vs}}$.
    Let~$\Delta:=\mA\vx$, $\widehat{\Delta}:=\mA\hat{\vx}$, and $\delta\Delta:=\mA\vz$.
    Define the vectors
    \[
    \vzeta := \sqrt{\vw}\odot \Delta, \qquad \beta := \sqrt{w}\odot \delta\Delta,
    \]
    so that, by Lemma~\ref{lem:switching_gradient},
    \(
    \nabla\phi(\vs) = -\,\vzeta\odot\vzeta
    \)
    and
    \(
    \widehat{\nabla}\phi(\vs) = -\,(\vzeta+\vbeta)\odot(\vzeta+\vbeta).
    \)
    Hence the error is
    \[
    h := \widehat{\nabla}\phi(\vs)-\nabla\phi(\vs) \;=\; -\,2\,\vzeta\odot\vbeta \;-\; \vbeta\odot\vbeta.
    \]
    
    From the Hessian factorization proved in Section \ref{sec:lemma:hessian_operator_norm_bound},
    \[
    \nabla^2\phi(\vs)\;=\;2\,\diag(\vzeta)\,\mP\,\diag(\vzeta),
    \quad
    \mP := \mW^{1/2}\mA\,\mL_{\vs}^\dagger \mA^\top \mW^{1/2},
    \]
    where $\mW=\diag(\vw)$ and $\mP$ is an orthogonal projector on $\Span(\mW^{1/2}\mA)$ with $\|\mP\|_{\mathrm{op}}=1$.
    Therefore, for any $\vu\in\R^m$,
    \[
    \|\vu\|_{\nabla^2\phi(\vs)^\dagger}
    =\tfrac{1}{\sqrt{2}}\;\big\|\mP^{\dagger/2}\,\diag(\vzeta)^{-1}\vu\big\|_2 .
    \]
    Applying this to $u=h$ and using the triangle inequality,
    \[
    \|\vh\|_{\nabla^2\phi(\vs)^\dagger}
    \le \tfrac{1}{\sqrt{2}}\Big( 2\,\|\mP^{\dagger/2}\vbeta\|_2
    +\big\|\mP^{\dagger/2}\,\diag(\vzeta)^{-1}(\vbeta\odot\vbeta)\big\|_2 \Big).
    \]
    Because $\vbeta=\mW^{1/2}\mA\vz$ lies in the range of $\mP$ (so $\mP^\dagger=\mP$ on that range),
    \[
    \|\mP^{\dagger/2}\vbeta\|_2^2 \;=\; \vbeta^\top \mP^\dagger \vbeta \;=\; \vz^\top \mL_{\vs}\vz \;=\; \|\vz\|_{\mL_{\vs}}^2.
    \]
    For the quadratic term, use
    \(
    \|\mP^{\dagger/2}\diag(\vzeta)^{-1}(\vbeta\odot\vbeta)\|_2
    \le \|\mP^{\dagger/2}\vbeta\|_2 \cdot \max_e |\beta_e|/|\zeta_e|.
    \)
    By Lemma~3 (Cauchy–Schwarz for voltage differences),
    $|\delta\Delta_e|\le \sqrt{\varrho_e}\,\|\vz\|_{\mL_{\vs}}$ and $|\Delta_e|\le \sqrt{\varrho_e}\,\|\vx\|_{\mL_{\vs}}$,
    so whenever $\zeta_e\neq 0$,
    \[
    \frac{|\beta_e|}{|\zeta_e|}=\frac{\sqrt{w_e}\,|\delta\Delta_e|}{\sqrt{w_e}\,|\Delta_e|}\le \frac{\|\vz\|_{\mL_{\vs}}}{\|\vx\|_{\mL_{\vs}}}.
    \]
    Thus,
    \[
    \|\vh\|_{\nabla^2\phi(\vs)^\dagger}
    \le \Big(\sqrt{2}\,\frac{\|\vz\|_{\mL_{\vs}}}{\|\vx\|_{\mL_{\vs}}}
    +\tfrac{1}{\sqrt{2}}\,\frac{\|\vz\|_{\mL_{\vs}}^2}{\|\vx\|_{\mL_{\vs}}^2}\Big)\,\|\vx\|_{\mL_{\vs}}.
    \]
    Finally, observe that
    \(
    \|\nabla\phi(\vs)\|_{\nabla^2\phi(\vs)^\dagger}
    =\tfrac{1}{\sqrt{2}}\|\mP^{\dagger/2}\vzeta\|_2
    =\tfrac{1}{\sqrt{2}}\|\vx\|_{\mL_{\vs}}
    \)
    (since $\vzeta=\mW^{1/2}\mA\vx$).
    With $\|\vz\|_{\mL_{\vs}}\le \varepsilon\|\vx\|_{\mL_{\vs}}$ this yields
    \[
    \|h\|_{\nabla^2\phi(\vs)^\dagger}
    \le (2\varepsilon+\varepsilon^2)\,\|\nabla\phi(\vs)\|_{\nabla^2\phi(\vs)^\dagger}
    \ \le\ 3\varepsilon\,\|\nabla\phi(\vs)\|_{\nabla^2\phi(\vs)^\dagger}
    \]
    for $\varepsilon\in(0,1]$, completing the proof.
    \end{proof}

\subsection{Proof of Theorem \ref{thm:backbone-gsc}}
    \label{apdx:proof:backbone-gsc}
    \begin{proof}[Proof of Theorem \ref{thm:backbone-gsc}]
    Our proof follows the general pattern of~\cite{sun2019generalized} and relies on explicit formulas for the derivatives of $\phi$.  First, recall the definition of a $(M,\nu)$--generalized self\nobreakdash--concordant function from~\cite[Definition~1.3]{carderera2024scalable}: a closed convex function $f\in C^3(\mathrm{dom}(f))$ with open domain is $(M,\nu)$--generalized self\nobreakdash--concordant if for all $x$ and all directions $u,w$,
    \begin{equation*}
        \bigl|\langle D^3 f(x)[w]u, u\rangle\bigr| \;\le\; M\,\|u\|_{\nabla^2 f(x)}^2\;\|w\|_2^{\nu-2}\,\|w\|_2^{3-\nu},
    \end{equation*}
    where $D^3 f(x)[w]$ is the third Fr\'echet derivative of $f$ and $\|v\|_{\nabla^2 f(x)}^2=v^{\top}\nabla^2 f(x)\,v$ is the local norm.  Setting $\nu=2$ yields the bound $M\,\|u\|_{\nabla^2 f(x)}^2\,\|w\|_2$.
    
    We now derive expressions for the derivatives of $\phi$.  Define $x=L(s)^{-1}d\in\R^n$, which satisfies $L(s)x=d$ and $x\perp\mathbf{1}$.  Let $V(h)=A^{\top}\operatorname{diag}(\vw\odot h)A$ for $h\in\R^m$ so that $D L(s)[h]=V(h)$.  Differentiating $x$ gives $D x[s][h]= -L(s)^{-1} V(h) x$.  A standard computation using these identities (see, e.g., our earlier derivation or~\cite[Section~1]{carderera2024scalable}) yields the first three directional derivatives of $\phi$:
    \begin{align*}
        D\phi[s][h] &= -\,x^{\top} V(h)\,x,\\
        D^2\phi[s][u,v] &= 2\,x^{\top} V(u)\,L(s)^{-1} V(v)\,x,\\
        D^3\phi[s][w,u,u] &= -6\,x^{\top} V(u)\,L(s)^{-1} V(w)\,L(s)^{-1} V(u)\,x.
    \end{align*}
    Since $V(h)$ is symmetric, the second derivative can be written more compactly as $\nabla^2\phi(s)=2\,\operatorname{diag}(\zeta)\,P\,\operatorname{diag}(\zeta)$, where $\zeta_e=\sqrt{w_e}\,a_e^{\top}x$ and $P = \omega^{1/2}A\,L(s)^{-1}A^{\top}\omega^{1/2}$ is an orthogonal projector (its square equals itself) as described in~\cite{carderera2024scalable}.  Therefore $\|u\|_{\nabla^2\phi(s)}^2=2\,(\operatorname{diag}(\zeta)u)^{\top} P\,\operatorname{diag}(\zeta)u\ge 0$.
    
    To bound the third derivative, we whiten the expression.  Set $y=L(s)^{-1/2} V(u) x$ and $B=L(s)^{-1/2} V(w)\,L(s)^{-1/2}$.  The matrix $B$ is symmetric positive semidefinite.  With this notation,
    \begin{equation*}
        D^3\phi[s][w,u,u] \;=\;-6\,y^{\top}B\,y.
    \end{equation*}
    Since $\|y\|_2^2 = y^{\top} y = x^{\top} V(u)\,L(s)^{-1} V(u)\,x = \tfrac{1}{2}\,u^{\top}\nabla^2\phi(s)\,u$, we obtain
    \begin{equation}
        \bigl|D^3\phi[s][w,u,u]\bigr| \;\le\; 6\,\|B\|_{\mathrm{op}}\,\|y\|_2^2 \;=\;3\,\|B\|_{\mathrm{op}}\,\|u\|_{\nabla^2\phi(s)}^2.
        \label{eq:third-phi-bound}
    \end{equation}
    It remains to bound the operator norm $\|B\|_{\mathrm{op}}$ uniformly over $s\in\mathcal{D}_T$.  For any vector $y\perp\mathbf{1}$ we have the Rayleigh quotient
    \begin{equation*}
        \frac{y^{\top} V(w)\,y}{y^{\top} L(s)\,y} \;=\; \frac{\sum_{e=1}^m |w_e|\,w_e\,(a_e^{\top}y)^2}{y^{\top}L(s)\,y}.
    \end{equation*}
    Because $s_e=1$ for all $e\in T$, one has $L(s)\succeq L_T$ uniformly on $\mathcal{D}_T$.  For every edge $e\in E$, Cauchy--Schwarz in the $L_T$\nobreakdash--energy inner product yields
    \begin{equation*}
        (a_e^{\top}y)^2 \;=\; \langle L_T^{\dagger}a_e,\,y\rangle_{L_T}^2 \;\le\; \|L_T^{\dagger}a_e\|_{L_T}^2 \,\|y\|_{L_T}^2 \;=\; \rho_T(e)\,\bigl(y^{\top}L_T\,y\bigr) \;\le\; \rho_T(e)\,y^{\top}L(s)\,y.
    \end{equation*}
    Summing over $e$ leads to
    \begin{equation*}
        y^{\top} V(w)\,y \;\le\; \Bigl(\sum_{e=1}^m |w_e|\,w_e\,\rho_T(e)\Bigr) y^{\top} L(s)\,y \;\le\; \bigl\|\vw\odot \vrho_T\bigr\|_2 \cdot \|w\|_2\;y^{\top} L(s)\,y.
    \end{equation*}
    Taking the supremum over all $y\perp\mathbf{1}$ gives
    \begin{equation*}
        \bigl\|L(s)^{-1/2} V(w)\,L(s)^{-1/2}\bigr\|_{\mathrm{op}} \;\le\; \bigl\|\vw\odot \vrho_T\bigr\|_2\;\|w\|_2
    \end{equation*}
    uniformly for all $s\in\mathcal{D}_T$.  Substituting this bound into~\eqref{eq:third-phi-bound} yields
    \begin{equation*}
        \bigl|D^3\phi[s][w,u,u]\bigr| \;\le\; 3\,\bigl\|\vw\odot \vrho_T\bigr\|_2\,\|w\|_2\,\|u\|_{\nabla^2\phi(s)}^2,
    \end{equation*}
    which exactly matches the definition of $(M,\nu)$\nobreakdash--generalized self\nobreakdash--concordance with $\nu=2$ and $M=3\|\vw\odot \vrho_T\|_2$.  This bound is independent of $s$, completing the proof.
    \end{proof}
    
    \begin{remark}
    The constant $M$ depends on the geometry of the fixed spanning tree $T$ through its effective resistances $\rho_T(e)$.  A tree with small stretch yields a smaller $M$ and thus stronger curvature control.  When $\omega\equiv 1$, one has the bound $\|\vw\odot \vrho_T\|_2\le \sqrt{m}\,\max_{e\in E} \rho_T(e)$, though in practice low\nobreakdash--stretch spanning trees can make $\|\omega\odot \vrho_T\|_2$ much smaller.
    \end{remark}

    \begin{remark}
        Let~$\vs_0 \in \c{0,1}^m$ be such that~$\p{\vs_0}_e=1$ for all~$e \in T.$ Observe that
        \[
        M = 3\norm{\vw \circ \vrho_T}_2 \leq 3\norm{\vrho_T \circ \vw \circ \vone_{E\setminus T}}_\infty\norm{\vell_T}_1\leq 3n\cdot\max_{e \in E \setminus T} w_e \va_e^\T\mL_{\vs_0}^\dagger\va_e.
        \]
        where the first inequality is by Holder's inequality, and the second inequality is by Foster's theorem~\cite{foster1949average}. 
    \end{remark}
\subsection{Proof of Theorem \ref{thm:bern_rounding_main}}
\label{apdx:thm:bern_rounding_main}
\begin{proof}[Proof of Theorem \ref{thm:bern_rounding_main}]
For each edge $e$, let
\[
\mY_e := \xi_e w_e \va_e\va_e^\T, \qquad \xi_e \sim \mathsf{Ber}(\bar{s}_e), \ \mathsf{independent},
\]
so that $\mL_{\vstilde} = \sum_{e \in E} \mY_e$, and its expectation is the ``smoothed'' Laplacian
\[
  \Expec[]{\mL_{\tilde{\vs}}} \;=\; \sum_{e\in E} \bar{s}_e \, w_e \, \va_e \va_e^\T =: \mL_{\bar{s}}.
\] 
Since, for all $e$, $\norm{\va_e}_2^2 =2$ and $\mY_e \succeq \vzero$ and has rank one with the only nonzero eigenvalue $2w_e$, we have the following PSD upper bound:
\[
\vzero \preceq \mY_e \preceq 2 w_e \mId_n \preceq 2 w_{\sf max} \mId_n  =: R \mId_n \qquad \forall e \in E.
\]
Now, by connectivity of the backbone $T \subseteq E,$ and because $\bar{s}_e =1$ for all $e \in T$, we have $\mL_{\vsbar} \succeq \mL_{T} \succ 0$. On  $\Span(\vone)^\perp$, denote
\[
\mu_{\sf min} = \lambda_2(\mL_{\vsbar}), \qquad \mu_{\sf max}:=\lambda_{\sf max}(\mL_{\vsbar}).
\]
We will invoke the following lightly modified version of \cite[Corollary 5.2]{Tropp_2011}.
\begin{lemma}[Matrix Chernoff for PSD sums]
\label{lem:psd_chernoff}
Let $\{\mY_i\}_{i=1}^n$ be independent random PSD matrices acting on a fixed $d$-dimensional subspace $U\subseteq\R^n$, with $\Expec[]{\sum_i \mY_i}=\mPi$ and $\opnorm{\mY_i}\le R$ almost surely. Then, for any $\epsilon\in(0,1)$,
\begin{subequations}
    \begin{align}
        \Pr\s{\sum_{i}\mY_i \succcurlyeq \p{1+\epsilon} \mPi} \leq d \exp\p{-\frac{\epsilon^2 \mu_{\sf max}}{3 R}},\\
        \Pr\s{\sum_{i}\mY_i \preccurlyeq \p{1-\epsilon} \mPi} \leq d \exp\p{-\frac{\epsilon^2 \mu_{\sf  min}}{2 R}}.
    \end{align}
\end{subequations}
\end{lemma}

By Lemma \ref{lem:psd_chernoff}, the union bound, and the fact that $\mu_{\sf max} \geq \mu_{\sf min},$ we obtain, with probability at least $1-\delta$,
\[
(1-\epsilon)\,\mL_{\vsbar} \;\preccurlyeq\; \mL_{\tilde{\vs}} \;\preccurlyeq\; (1+\epsilon)\,\mL_{\vsbar}, \qquad \mathsf{on} \; \Span(\vone)^\perp,
\]
provided that
\[
\epsilon \geq \sqrt{\frac{3R\log(d/\delta)}{\mu_{\sf min}}} = C \sqrt{\frac{2 w_{\sf max} \log((n-1)/\delta)}{\lambda_2(\mL_{\vsbar})}}, \qquad C := \sqrt{3}.
\]
This is the advertised bound \eqref{eq:bern_spectral} with constant $C = \sqrt{3}.$ A standard scalar Bernoulli concentration argument completes the proof.
\end{proof}

\subsection{Proof of Theorem \ref{thm:fw_gap_convergence}}
\ifapdxthms
\begin{theorem}
    \label{thm:fw_gap_convergence}
    Let $\setS \subseteq \s{0,1}^m$ be a convex set of almost surely connected switching distributions, and let $\vs_t \in \setS$ be the $t$-th Frank-Wolfe iterate of the switching probabilities in the Frank-Wolfe algorithm. Set
    \begin{equation}
    \label{eq:lmo_vector}
    \vv_t^\star \in \argmin_{\vv \in S} \ \ip{\nabla\phi(\vs_t)}{\vv},
    \end{equation}
    Then, if
    \[
    \ip{\nabla \phi(\vs_t)}{\vs_t - \vv_t^\star} \leq \tau \cdot \phi(\vs_t),
    \]
    then
    \[
    \phi(\vs_t) - \phi(\vs_\star) \leq \frac{\tau}{1-\tau} \phi(\vs_\star)\; \iff \; \phi(\vs_t) \leq \frac{1}{1-\tau}\phi(\vs_\star),
    \]
    where $\vs_\star$ is a global minimizer of the convex function $\phi(\cdot)$ over $S$.
\end{theorem}
\fi
\begin{proof}[Proof of Theorem \ref{thm:fw_gap_convergence}]
    For the $t$-th iterate $\vs_t \in \setS$, define the \textit{Frank-Wolfe gap} $g : \setS \to \R$ as
    \begin{align*}
    g(\vs_t) := \max_{\vv \in \setS} \ip{\nabla \phi(\vs_t)}{\vs_t - \vv}
    = \ip{\nabla \phi(\vs_t)}{\vs_t} - \min_{\vv \in \setS} \ip{\nabla \phi(\vs_t)}{\vv}
    = \ip{\nabla \phi(\vs_t)}{\vs_t - \vv_t^\star},
    \end{align*}
    where $\vv_t^\star$ is as in \eqref{eq:lmo_vector}. By convexity,
    $
    \phi(\vs_\star) \geq \phi(\vs_t) + \ip{\nabla\phi(\vs_t)}{\vs_\star - \vs_t},
    $
    hence
    \[
    \phi(\vs_t) - \phi(\vs_\star) \leq -\ip{\nabla \phi(\vs_t)}{\vs_\star - \vs_t} \leq \max_{\vv \in \setS}\ip{\nabla\phi(\vs_t)}{\vs_t - \vv} = g(\vs_t) \leq \tau\phi(\vs_t).
    \]
    Rearranging gives $(1-\tau)\phi(\vs_t) \leq \phi(\vs_\star).$
\end{proof}

\subsection{Proof of Corollary \ref{thm:alpha_certificate}}
\begin{proof}[Proof of Corollary \ref{thm:alpha_certificate}]
    It suffices to show that the LMO step for the set $S_q(\vs_0)$ is as given in \eqref{eq:specific_lmo_entrywise}, and then invoke Theorem \ref{thm:fw_gap_convergence}. We now establish this.
    
    First, note that we trivially must have~$\vv^\star(\vs)_e =1$ for all~$e \in \mathsf{supp}(\vs_0)$ by construction. Further, by Lemma~\ref{lem:switching_gradient}, for any iteration $t$ we have that the congestion gradient entries are non-positive, as
    \[
    \forall e \in E, \quad \p{\nabla \phi(\vs_t)}_e = -w_e \abs{\ip{\va_e}{\mL_{\vs_t}^\dagger\vd}}^2 \leq 0 \qquad \forall \vs_t \in S_q(\vs_0), \ \vd \perp \vone;
    \]
    hence, it suffices to minimize the linear form~$\ip{\nabla(\vs_t)}{\vv}$ over~$\c{\vv\::\: \norm{\vv \circ \p{\vone - \vs_0}}_1 \leq q-|T|}$ by selecting the enties with the largest magnitude (most negative values) up to the remaining budget $q-|T|.$ The certificate then immediately follows from Theorem \ref{thm:fw_gap_convergence} with $\tau = \alpha/(1+\alpha).$
\end{proof}

\subsection{End-to-end guarantee \& runtime complexity in $\ell_{2\to2}$ space}
\label{apdx:end-to-end-FW-l2}
\begin{proof}[Proof of Theorem \ref{thm:FW-rounded-output-convergence}]
    Run the monotone Frank-Wolfe method with step sizes $\eta_t = 2/(t+2)$, and for $t =0,1,\ldots,$ let $\vs_t \in \s{0,1}$ be the incumbent fractional solution. 
    By standard convergence arguments, we have
    \[
    t \geq \ceil*{\frac{2LD^2}{\alpha}} \implies \phi(\vs_t) \leq \p{1+\alpha}\phi(\vs_\star)
    \]
    for any minimizer $\vs_\star \in \argmin_{\vs \in S_q(\vs_0)} \phi(\vs).$ 
    
    To see this, note that by assumption and by Lemma \ref{lemma:hessian_operator_norm_bound}, we have that $L \leq 2$. Moreover, the diameter of the feasible region $S_q(s_0)$ is
     \[
     D := \sup_{y,z \in S_q(s_0)} \, \norm{y-z}_2 = \sup_{y,z \in S_q(s_0)}\sqrt{\sum_{e \in E\setminus S_0} (y_e - z_e)^2} = \sqrt{2(q-n+1)} \leq \sqrt{q}.
     \]
    Thus, standard Frank-Wolfe gap guarantees imply that
    \[
    g(\vs_t) \leq \frac{2LD^2}{t+2} \leq \alpha,
    \]
    and Corollary~\ref{thm:alpha_certificate}/Theorem~\ref{thm:fw_gap_convergence} gives
    \[
    \phi(\vs_t) \leq \p{1+\alpha}\phi(\vs_\star).
    \]

    Now, we want to apply Theorem \ref{thm:bern_rounding_main}. We first must show that $\norm{\vs_t}_1 \leq q$ for all $t$ by induction. The base case is true since $\vs_0$ is a feasible point. Now, for the induction step, suppose that $\sum_e s_e^{(t)} \leq q$. Then, at iteration $t+1$, we have
    \begin{align*}
    \sum_e s_e^{(t+1)} &= \sum_e \p{(1-\gamma_t) s_e^{(t)} + \gamma_t v_{e}^\star(s^{(t)})} \\
    &= (1-\gamma_t) \sum_e s_e^{(t)} + \gamma_t \sum_e v_{e}^\star(s^{(t)}) \\
    &\leq (1-\gamma_t) q + \gamma_t q \\
    &= q.
    \end{align*}
    Thus, by induction, we have that $\sum_e s_e^{(t)} \leq q$ for all $t$.

    By Theorem \ref{thm:bern_rounding_main} and the stated assumptions, we have
    \[
    \epsilon \leq \sqrt{\frac{3R\log\p{(n-1)/\delta}}{\lambda_2(\mL_{\vs_t})}} \leq \sqrt{3R\log((n-1)/\delta)} \lesssim \sqrt{\log(n)/\delta},
    \]
    where the second inequality is by assumption that $\lambda_2(\mL_{\vs_0}) \geq 1$, so $\lambda_2(\mL_{\vs_t}) \geq 1$ for all $t\geq1$, also.

    Finally, note that each FW step costs one Laplacian solve + one $O(m)$ pass over the edges + top-$q$ selection, for $t = \Theta(q/\alpha)$ iterations. Thus, the total running time is $\tilde{O}(mq/\alpha).$ This completes the proof.
\end{proof}

\section{Additional Analyses}
\label{sec:complexity}
In this section we present further analyses that we were unable to fit in the main submission.

\subsection{Satisfying the budget with high probability}
\label{apdx:satisfying_the_budget_high_prob}
Here, we provide rigorous modifications of the Bernoulli parameters to ensure that the constraint $\norm{\vstilde}_1 \leq q$ is satisfied with high probability after rounding.
\begin{proposition}
    \label{prop:bern_constraint}
    Let~$\vs_0 \in \c{0,1}^m$ be a connected template and let~$\vs \in S_q(\vs_0)$ be any feasible switching distribution. Fix a failure probability~$\delta \in (0,1)$ and~$0<\gamma\leq q-\abs{T}.$ Define the shrinkage parameter
    \[
    \theta := \begin{cases}
        1 &\mathsf{if} \ \sum_{e \in E \setminus T} s_e \leq q - \abs{T}-\gamma\\
        \frac{q-\abs{T}-\gamma}{\sum_{e \notin T}s_e} &\ow,
    \end{cases}
    \]
    and set~$s_e' =1$ for~$e \in T$ and~$s_e' = \theta s_e$ otherwise.  Then, if~$\gamma \geq \sqrt{2\p{q-\abs{T}}\log(1/\delta)},$ and we sample a rounded vector~$\vstilde \in \c{0,1}^m$ as a vector Bernoulli random variables with~$\p{\vstilde}_e \sim \mathsf{Ber}(s_e'),$ we have that
    $
    \Pr\p{\norm{\vstilde}_1 \leq q} \geq 1-\delta.
    $
\end{proposition}
\begin{proof}
    Let~$Y:=\sum_e \tilde{s}_e = \sum_{e \in T} \tilde{s}_e + \sum_{e \in E\setminus T} \tilde{s}_e.$ For~$e \in T,$ $s_e'=1$, so $\sum_{e \in T} \tilde{s}_e=\abs{T}.$ Moreover, for~$e \notin T$, $\E \tilde{s}_e=s_e'$, and by construction~$\mu:=\Expec[]{\sum_{e\notin T}\tilde{s}_e}=\sum_{e\notin T} s_e' \leq q - \abs{T}-\gamma.$ Apply the scalar Bernstein inequality to~$S := \sum_{e\notin T} \p{\tilde{s}_e - \E \tilde{s}_e},$ and, for any $t>0$, we have 
    \[
    \Pr\p{S\geq t} \leq \exp\p{\frac{-t^2}{2\p{\mu+t/3}}}.
    \]
    With $t=\gamma,$ we get 
    \[
    \Pr\p{Y \geq \abs{T}+\mu+\gamma} \leq \exp\p{-\frac{\gamma^2}{2\p{\mu+\gamma/3}}}\leq\exp\p{-\frac{3\gamma^2}{8\p{q-\abs{T}}}},
    \]
    where the final inequality is due to the fact that~$\mu\leq q-\abs{T}-\gamma.$ Choose~$\gamma\geq\sqrt{2\p{q-\abs{T}}\log(1/\delta)}\implies \mu+\gamma/3\leq q-\abs{T}$ to get the bound~$\Pr\p{Y\geq\abs{T} + \mu + \gamma} \leq \exp\p{-\frac{\gamma^2}{2\p{q-\abs{T}}}}.$ That is, choosing $\gamma \geq \sqrt{2\p{q-\abs{T}}\log(1/\delta)}$ yields $\Pr\p{Y\geq q} \leq \delta.$
\end{proof}

    \subsection{Alternative complexity analysis in $\ell_{1\to\infty}$ space}
    \label{sec:linf-l1-space}
    The runtime analysis in Appendix~\ref{apdx:end-to-end-FW-l2} yields the optimal asymptotic rate; we can achieve the same asymptotic complexity in a different geometry, which may have certain advantages. For a fixed budget~$q \in \N_+,$ define the norm
    \[
    \norm{\cdot}_{\hexagon} := \max\c{\norm{\cdot}_1,\, q\norm{\cdot}_\infty};
    \]
    with corresponding norm ball  $\setB_{\hexagon} \subseteq \R^m$ as 
    \[
    \setB_{\hexagon} := \c{\vu \in \R^m\;:\; \norm{\vu}_{\hexagon} \le 1} = \c{\vu \in \R^m \;:\; \max\c{\norm{\vu}_1,\; q\norm{\vu}_\infty} \leq 1}.
    \]
    This family of norms describes a polytopal geometry~\cite{deza_polytopal_2021}; a subset of the unit hypercube with cut corners, which is related to Eulerian numbers. 
    \begin{proposition}
        Let~$\vu \in \R^m$, and let~$\abs{u}_{[1]}\geq\abs{u}_{[2]}\geq\dots \geq \abs{u}_{[m]}$ denote the magnitude of the entries of $\vu$ sorted in non-increasing order. For a fixed $q \in \N_+$, the dual norm of $\norm{\cdot}_{\hexagon}$ is given as
        \[
        \norm{\vu}_{\hexagon,\star} = \sup_{\norm{\vy}_1\leq 1,\; \norm{\vy}_\infty \leq 1/q}\; \ip{\vu}{\vy} = \frac{1}{q} \sum_{i \in q} \abs{u}_{[i]}, 
        \]
        i.e., the average of the $q$ largest absolute coordinates.
    \end{proposition}
    \begin{proof}
        By definition, for any $\vu \in \R^m$ the dual norm of $\norm{\cdot}_{\hexagon}$ is
        \[
        \norm{\vu}_\star := \sup_{\vy : \norm{\vy}_{\hexagon} \leq 1}\; \ip{\vy}{\vu}.
        \]
        This is equivalent to the constrained optimization problem
        \[
        \norm{\vu}_{\hexagon,\star} = \sup_{\vy \in \R^m} \; \ip{\vy}{\vu} \quad \st \quad \sum_{k=1}^m y_k \leq 1, \quad 0 \leq y_k \leq 1/q,\ k=1,\ldots,m.
        \]
        This is knapsack problem with identical bounds; the optimal solution is clearly~$y_i = 1/q$ for indices~$i=1,\ldots,q$ corresponding to $\abs{u}_{[1]}\geq \abs{u}_{[2]}\geq\ldots\geq\abs{u}_{[q]}.$ 
    \end{proof}

    Now, we relate our choice of norm to $\norm{\cdot}_2$.
    \begin{lemma}
    \label{lemma:l2-hexagon-relation}
        We have 
        \[
        \norm{\vu}_{\hexagon,\star} \leq \frac{1}{\sqrt{q}}\norm{\vu}_2.
        \]
    \end{lemma}
    \begin{proof}
        By Cauchy-Schwarz,
        \[
        \norm{\vu}_{\hexagon,\star} = \frac{1}{q} \sum_{i=1}^q \abs{u}_{[i]} = \frac{1}{q}\sum_{i=1}^m \abs{u}_{[i]} \cdot \Ind{i \leq q} =\frac{1}{q} \ip{\vone_{i\leq q}}{ \vu \circ\sgn(\vu)} \overset{(1)}{\leq} \frac{1}{\sqrt{q}} \norm{\vu}_2,
        \]
        where in step (1) we applied norm equivalence, namely~$\norm{\vu}_1 \leq \sqrt{m}\norm{\vu}_2$ for~$\vu \in \R^m.$ 
    \end{proof}

    We immediately yield the following result.
    \begin{lemma}
    \label{lemma:the_right_constants}
        Fix an integer $q \in \N_+$, and a connected template $\vs_0$. Define
        \[
        S_q(\vs_0) = \c{\vs \in \s{0,1}^m\; \norm{\vs}_1 \leq q,\; s_e =1 \; \forall e \in \mathsf{supp}(\vs_0)}.
        \]
        Assume that $\phi(\cdot)$ is $L$-smooth with respect to $\norm{\cdot}_2$. Then, for fixed demands $\vd\perp\vone$, the relaxed congestion~$\phi(\vs) = \vd^\T\mL_{\vs}^\dagger\vd$ is~$\beta$-smooth w.r.t. $\norm{\cdot}_{\hexagon}$, i.e.,
        \[
        \norm{\nabla \phi(\vs)-\nabla\phi(\vs')}_{\hexagon,\star} \leq \beta\norm{\vs-\vs'}_{\hexagon} \qquad \forall \vs,\vs' \in S_q(\vs_0),
        \]
        with 
        $
        \beta \leq L/q.
        $
    \end{lemma}
    \begin{proof}
        By applying Lemma \ref{lemma:l2-hexagon-relation} twice, we have
        \begin{align*}
            \norm{\nabla \phi(\vs) - \nabla \phi(\vs')}_{\hexagon,\star} \leq \frac{1}{\sqrt{q}} \norm{\nabla \phi(\vs) - \nabla \phi(\vs')}_2 \leq \frac{L}{\sqrt{q}} \norm{\vs - \vs'}_2\leq \frac{L}{q}\norm{\vs-\vs'}_{\hexagon}.
        \end{align*}
        This completes the proof.
    \end{proof}

\end{document}